\documentclass[reqno, 11pt]{amsart}
\usepackage{amssymb,amsmath,amsfonts}
\usepackage[margin=1in]{geometry}
\usepackage{dsfont}
\usepackage{color}
\usepackage{graphicx}
\usepackage{latexsym}
\usepackage{amsmath}
\usepackage{amssymb}
\usepackage{graphics}
\usepackage[dvips]{epsfig}
\usepackage{mathrsfs}
\usepackage{mathtools}
\usepackage{leqno}
\usepackage{stmaryrd,cite}

\usepackage{cases}
\usepackage{enumerate}
\usepackage[usenames,dvipsnames,svgnames]{xcolor}
\definecolor{refkey}{rgb}{1,0,0.5}
\definecolor{labelkey}{rgb}{0,0.4,1}
\usepackage{todonotes}
\usepackage{marginnote}

\presetkeys{todonotes}{fancyline, color=green!40}{}

\makeatletter
\renewcommand{\@todonotes@drawMarginNoteWithLine}{%
	\begin{tikzpicture}[remember picture, overlay, baseline=-0.75ex]%
	\node [coordinate] (inText) {};%
	\end{tikzpicture}%
	\marginnote[{
		\@todonotes@drawMarginNote%
		\@todonotes@drawLineToLeftMargin%
	}]{
		\@todonotes@drawMarginNote%
		\@todonotes@drawLineToRightMargin%
	}%
}
\makeatother
\usepackage{listings}
\usepackage{ifpdf}
\ifpdf
\usepackage[CJKbookmarks=true,
         hyperindex=true,
         pdfstartview=FitH,
         bookmarksnumbered=true,
         bookmarksopen=true,
         colorlinks=true,
         citecolor=blue,
         linkcolor=blue,
         urlcolor=blue,
         pdfborder=001,
         pdfauthor={},
         pdftitle={},
         pdfkeywords={},
         ]{hyperref}
\else
\usepackage[
         hypertex,
         hyperindex,
         linkcolor=blue,
         unicode,
         citecolor=blue%
         ]{hyperref}
\fi
\usepackage{cleveref}

\allowdisplaybreaks

\numberwithin{equation}{section}
\newtheorem{thm}{Theorem}[section]

\newtheorem{lem}[thm]{Lemma}
\newtheorem{prop}[thm]{Proposition}
\newtheorem{rmk}[thm]{Remark}

\newcommand{\be}{\begin{equation}}
\newcommand{\ee}{\end{equation}}

\newcommand{\bee}{\begin{equation*}}
\newcommand{\eee}{\end{equation*}}

\newcommand{\bse}{\begin{subequations}}
\newcommand{\ese}{\end{subequations}}

\newcommand{\bs}{\begin{split}}
\newcommand{\es}{\end{split}}

\begin{document}

\author{Hairong Liu$^{1}$}\thanks{$^{1}$School of Science, Nanjing Forestry University, Nanjing, China.
E-mail : hrliu@njfu.edu.cn}

\author{Hua Zhong$^{2}$}\thanks{$^{4}$Department of Mathematics, City University of Hong Kong, 83 Tat Chee Avenue, Kowloon Tong, Hong Kong.
E-mail: huazhong3-c@my.cityu.edu.hk}

\title[] {Global Solutions to the initial boundary problem of 3-D compressible Navier-Stokes-Poisson on  bounded domains}

\begin{abstract}

 The initial boundary value problems for compressible Navier-Stokes-Poisson  is considered on a bounded domain in $\mathbb{R}^3$ in this paper. The global existence  of smooth solutions near a given steady state for compressible Navier-Stokes-Poisson with physical boundary conditions is established with the exponential stability.  An important feature is that the steady state (except velocity) and the background profile are allowed to be of  large variation.

\end{abstract}

\maketitle

\section{Introduction}

It is well-known that the compressible Navier-Stokes-Poisson (NSP) system  consists of the Navier-Stokes equations coupled with the self-consistent Poisson equations, which is used in the simulation of the motion of charged particles (electrons or holes, see \cite{[23]} for more explanations). In three dimensional space, the NSP system of one carrier type takes the following form
\begin{eqnarray}{\label{prob1}}
\left\{
\begin{array}{llll}
\rho_{t}+\mbox{div}(\rho u)=0,\\[2mm]
\rho\left(u_{t}+(u\cdot\nabla)u\right)-\mu\Delta u-(\mu+\lambda)\nabla\mbox{div}u+\nabla p=\rho\nabla\Phi,\\[2mm]
\Delta\Phi=\rho-\bar{\rho},
\end{array}
\right.
\end{eqnarray}
where  $\rho>0$, $u=(u^1,u^2,u^3)$ and $p$ denote the density, the velocity field of charged particles, the pressure, respectively.
The self-consistent electric potential $\Phi=\Phi(x,t)$ is coupled with the density through the Poisson equation.
The pressure $p$ is expressed by
\begin{equation}{\label{pressure}}
p(\rho)=\rho^{\gamma},
\end{equation}
where $\gamma\geq1$ is a constant.  As usual, the constant viscosity coefficients $\mu$ and $\lambda$ should satisfy the following physical conditions
$$\mu>0,\,\,\,\lambda+\frac{2}{3}\mu\geq0.$$
And $\bar{\rho}=\bar{\rho}(x)>0$ is the background profile, the sum of the background ion density and the net density of impurities, which is assumed to be given and immobile.

The object of this paper is to investigate the global existence and long-time behavior of the solutions to the initial boundary value (IBV) problem  of (\ref{prob1}) in $(t,x)\in [0,+\infty)\times\Omega$ , where $\Omega\subset\mathbb{R}^3$ is
 a smooth bounded domain $\Omega\subset\mathbb{R}^3$,  with the following initial condition
\begin{equation}{\label{initial}}
(\rho, u)(x,t=0)=(\rho_0,u_0)(x),
\end{equation}
and boundary condition
\begin{equation}{\label{boundary}}
u|_{\partial\Omega}=0,\,\,\,\nabla\Phi\cdot\nu\mid_{\partial\Omega}=0,
\end{equation}
where $\nu$ is the unit outer normal  to $\partial\Omega$.

The large-time behavior of solutions to compressible Navier-Stokes equations has been investigated extensively in \cite{Deckelnick1992,Deckelnick1993,DUYZ2007,DLUY2007,HoffZ1995,HoffZ1997,KK2002,KK2005,Kobayashi2002,KS1999,matsumura1979,matsumura1980,MatsumuraNi1983}. While for the Cauchy problem (initial value problem without boundaries) of the above Navier-Stokes-Poisson system, recently the decay rate of solutions was studied, see \cite{HL2010,HMW2003,huang2012,jang2013,ju2018,LMAT2010, LMM2002, LNX1998, Tan2013, WW2010, wang, ZLZ2011} for instance and the references therein, which has been proved that the electric field plays an important role on the large time behavior of solution and the solution will approach to constant state at an algebraic decay rate. For the compressible  Navier-Stokes-Poisson of self-gravitating fluids,  free boundary problems are a very active research subject, for which the  gravitational force plays a different role, compared with the electric forces. One may refer to \cite { luo2014,luo2016,luoadv2016} for this topic.

In the presence of physical boundaries, the regularity of solutions near boundaries is a very subtle and important issue for fluids and plasma equations. For this,  the classical global existence of smooth solutions to the initial boundary value problem for 3-D compressible Navier-Stokes equations was due to  Matsumura \& Nishida\cite{MatsumuraNi1983} for initial date being small perturbations of constant states  without convergence rate.  Recently, \cite{LLZ2020} has shown that the radially symmetric solutions exist globally to compressible Navier-Stokes-Poisson equations with the large initial data on a domain exterior to a ball in any dimensional space, moreover, the global existence of smooth solution near a given constant steady state for 3-D  compressible NSP equations with damping term on an exterior domain has been established with the exponential stability.

It should be noted  that Guo \&  Strauss\cite{GStrauss2005} established the asymptotic stability
of the stationary solution of Euler-Poisson equations  for the general doping profile in a bounded domain. For the initial boundary problem of Navier-Stokes-Possion equations considered in this paper, the viscous term creates difficulties in the analysis in the presence of physical boundaries, compared with the inviscid case considered in \cite{GStrauss2005}.  Inspired by \cite{MatsumuraNi1983,LLZ2020,GStrauss2005}, the steady states about space variable, instead of the constant steady state, begin to be considered. Compared with \cite{LLZ2020}, we prove in this paper that the solution exists globally and stabilizes exponentially to the steady state without the damping term for bounded domains.

Now we state the main result of this paper:
\begin{thm}\label{thm1}
Let $\Omega$ be a smooth bounded  domain in $\mathbb{R}^3$ and $\bar{\rho}(x)>0$ be a smooth function on $\bar{\Omega}$.
Let $\tilde{\rho}(x)>0$, $\tilde u\equiv 0$ and $\tilde{\Phi}(x)$ be a smooth steady state solution of (\ref{prob1}) such that $\frac{\partial \tilde{\Phi}}{\partial \nu}|_{\partial\Omega}=0$. Then there exists a constant $\delta>0$ such that if the initial data $(\rho_0, u_0)$ satisfies
	\begin{equation}\label{condrho}
	\int_\Omega(\rho_0-\bar{\rho})dx=0
	\end{equation}
and
\begin{equation*}
\|(\rho_0-\bar{\rho},u_0)\|_3^{2}\leq\delta^2,
\end{equation*}
then there exists a smooth global solution $(\rho,u,\Phi)(t,x)$ to the initial boundary value problem (\ref{prob1})-(\ref{boundary}).
 Moreover, there are positive constants
	$C$ and $\sigma$ such that
\begin{equation*}
\Big\{\|(\rho-\tilde{\rho},u,\nabla(\Phi-\tilde{\Phi}))\|_3^2+\|\rho_t\|_2^2+\|u_t\|_1^2\Big\}(t)\leq Ce^{-\sigma t}\Big\{
\|(\rho-\tilde{\rho},u,\nabla(\Phi-\tilde{\Phi}))\|_3^2+\|\rho_t\|_2^2+\|u_t\|_1^2\Big\}(0).
\end{equation*}
\end{thm}

\begin{rmk}
	An important feature of this paper is that the profile $\bar{\rho}(x)$ and steady state $\tilde{\rho}(x)$,  $\tilde{\Phi}(x)$ are allowed to be of large variation.
\end{rmk}
\begin{rmk}
	The condition (\ref{condrho}) persists in time, and  is the necessary condition of solvability of  the Poisson equation with Neumann boundary.
\end{rmk}

The rest of this paper is organized as follows. In Section 2, some useful elliptic estimates have been recalled firstly. Secondly, a steady state $(\tilde{\rho}, \tilde{u}, \tilde{\Phi})(x)$ of (\ref{prob1}) is established appropriately, which help us to  reconstruct the IBV problem for the perturbation veriables $(q,u,\phi)(t,x)$. Section 3 is devoted to show that  the global existence and exponential convergence to the steady state of smooth solutions.  Different from the Navier-Stokes equations, the electric field $\nabla\Phi$, located at the momentum equation, should be taken into account. The key to this is to consider the quantity $\nabla(\gamma\tilde{\rho}^{\gamma-2}q-\phi)$ to use the Stokes equation in Lemma \ref{lem10}, and apply Lemma \ref{lem-neu} to Poisson equation $(\ref{per1})_{3}$ and $(\ref{per1})_{5}$ to obtain the corresponding elliptic estimates. On the other hand, we cannot generally designate a coordinate system over all of $\Omega$ such that the directions are consistent with the normal and tangential directions on the boundary $\partial\Omega$. In order to overcome this difficult point, we divide the estimates of the solution into two parts: over the region away from the boundary and the near the boundary $\partial\Omega$, see Lemmas \ref{lemma6}-\ref{lem8}. In particular near the boundary, the estimates are quite involved. Using the local geodesic polar coordinates, we obtain the estimates for  tangential derivatives (Lemma \ref{lem7}), and then that for normal derivatives (Lemma \ref{lem8}).

{\bf Notations} Throughout this paper, $C$ will be used as a generic constant independent of time $t$.\\

(i)\  $\frac{df}{dt}=f_{t}+u\cdot\nabla f$ denotes the material derivative.\\

(ii)\ $\partial_{x_i}f=\frac{\partial f}{\partial x_i}=D_if$, $\frac{\partial^2 f}{\partial x_i\partial x_j}=\frac{\partial^2 f}{\partial x_j\partial x_i}=D_{ij}f$, $\frac{\partial^3 f}{\partial x_i\partial x_j\partial x_k}=D_{ijk}f$, $i,j,k=1,2,3$. Moreover,
	\begin{equation*}
		D^kf:=\{D^{\alpha}f| \,|\alpha|=k, k\in\mathbb{N}\}
		\end{equation*}
where $\alpha=(\alpha_1,\alpha_2,\alpha_3)$ is a multi-index, $D^\alpha :=\frac{\partial^{|\alpha|}f}{\partial x_1^{\alpha_1}\partial x_2^{\alpha_2}\partial x_3^{\alpha_3}}$, $|\alpha|=\alpha_1+\alpha_2+\alpha_3$  and $\alpha_i\geq0$.\\

(iii)\  $H^m$ is used to denote the Sobolev space with the following norm
\begin{equation*}
\|f\|_{{m}}\equiv\|f\|_{H^{m}(\Omega)}=\left(\sum_{l=0}^m\|D^{l}f\|^2\right)^{1/2}, \quad \mbox{and} \quad \|f\|\equiv\|f\|_{L^{2}(\Omega)}.
\end{equation*}

(vi)\ The Einstein's summation convention taken for $i,j,k=1,2,3$ will be used sometimes in this paper.

\begin{equation*}
\end{equation*}

\section{Preliminaries and the reformulation of the problem }
In this section, we first recall some estimates of elliptic equations, which will be used in the subsequent. Then
 we reformulate the problem in terms of perturbations.

The classical regularity theory for the Neumann problem of elliptic equation is as follows (see \cite{brezis}).
\begin{lem}\label{lem-neu}(Neumann problem)
Given an $f\in H^{k}(\Omega)$($k\in\mathbb{N}$) and a $g\in H^{k+1-1/2}(\partial\Omega)$ such that
\begin{equation*}
\int_{\Omega} fdx=\int_{\partial\Omega}gdS,
\end{equation*}
then there exists a  $v\in H^{k+2}(\Omega)$ satisfying
\begin{eqnarray*}
\left\{
\begin{array}{llll}
\Delta v=f,\quad \mbox{in}\quad \Omega,\\[2mm]
\frac{\partial v}{\partial\nu}=g,\quad \mbox{on}\quad \partial\Omega,
\end{array}
\right.
\end{eqnarray*}
and
\begin{equation*}
\|\nabla v\|_{k+1}\leq C\Big(\|f\|_{k}+\|g\|_{k+1-1/2}\Big).
\end{equation*}
\end{lem}

A steady state $\left(\tilde{\rho}(x), \tilde{u}(x), \tilde{\Phi}(x)\right)$ of (\ref{prob1}) can be obtained as:

\begin{prop}\label{prop1}
Let $\bar{\rho}(x)>0$ in $\Omega$. Then there exists a  smooth steady state solution
$(\tilde{\rho}(x), 0, \tilde{\Phi}(x))$ to the problem (\ref{prob1}) such that $\tilde{\rho}(x)>0$  in $\bar{\Omega}$.  \end{prop}

\begin{proof}
A steady state with $u\equiv0$ must satisfy the following equations:
\begin{equation*}
\nabla p(\rho)-\rho\nabla\Phi=0,\quad \Delta\Phi=\tilde\rho-\bar{\rho}
\end{equation*}
 with the boundary condition $\nabla\Phi\cdot\nu=0$ on $\partial\Omega$. Hence, the proof is the same as that for the Proposition 3 in \cite {GStrauss2005} and is omitted.
\end{proof}

Let $(\tilde{\rho}(x), 0, \tilde{\Phi}(x))$ be a given  stead state solution  by Proposition \ref{prop1}, that is
\begin{equation}\label{steady}
\nabla p(\tilde{\rho})-\tilde{\rho}\nabla\tilde{\Phi}=0,\quad \Delta\tilde{\Phi}=\tilde{\rho}-\bar{\rho}
\end{equation}
with
\begin{equation*}
\nabla\tilde{\Phi}\cdot\nu|_{\partial\Omega}=0,
\end{equation*}
and
\begin{equation*}
\int_{\Omega}\large(\tilde{\rho}(x)-\bar{\rho}(x)\large)dx=\int_{\partial \Omega}\nabla\tilde{\Phi}\cdot\nu ds=0,\  \tilde{\rho}(x)>0,\ x\in {\Omega}.\
\end{equation*}

Denote the perturbation $(q,u,\phi)(t,x)$ as
\begin{equation}
q(t,x)=\rho(t,x)-\tilde{\rho}(x), \quad u(t,x)=u(t,x),\quad \phi(t,x)=\Phi(t,x)-\tilde{\Phi}(x). \nonumber
\end{equation}
Then the initial boudary value problem for $(q,u,\phi)$ is
\begin{eqnarray}{\label{per1}}
\left\{
\begin{array}{llll}
q_{t}+\tilde{\rho}\mbox{div}u+u\cdot\nabla\tilde{\rho}=f^{0},\\[2mm]
\rho u_{t}
-\mu\Delta u-(\mu+\lambda)\nabla \mbox{div}u+\nabla(\gamma\tilde{\rho}^{\gamma-1} q)-\tilde{\rho}\nabla\phi-q\nabla\tilde{\Phi}
=f,\\[2mm]
\Delta\phi=q,\\[2mm]
u|_{\partial\Omega}=0,\\[2mm]
\nabla\phi\cdot\nu|_{\partial\Omega}=0,\\[2mm]
u(0,\cdot)=u_0,\quad q(0,\cdot)=q_0\equiv \rho_0-\tilde{\rho},
\end{array}
\right.
\end{eqnarray}
 where the nonlinear terms on the right-hand side are described as:
\begin{equation*}\label{f}
\begin{split}
&f^{0}=-\mbox{div} (qu),\\[2mm]
&f=-\rho(u\cdot\nabla) u +q\nabla\phi+\nabla h(q),
\end{split}
\end{equation*}
\begin{equation*}
 h(q)=(q+\tilde{\rho})^{\gamma}-\tilde{\rho}^{\gamma}-\gamma\tilde{\rho}^{\gamma-1}q=O(q^2).
\end{equation*}
The  equations of (\ref{per1}) are equivalent to the following equations:
\begin{equation}\label{per3-1}
L^{0}\equiv\frac{dq}{dt}+\mbox{div}(\tilde{\rho}u)=-q\mbox{div}u\equiv g^{0},
\end{equation}
\begin{equation}\label{per3-2}
L\equiv u_{t}-\mu\frac{1}{\tilde{\rho}}\Delta u-(\mu+\lambda)\frac{1}{\tilde{\rho}}\nabla (\mbox{div}u)
+\nabla(\gamma\tilde{\rho}^{\gamma-2}q)-\nabla\phi=g,
\end{equation}
where the nonlinear terms are given by
\begin{equation*}\label{g}
\begin{split}
g\equiv -(u\cdot\nabla) u+\mu \Big(\frac{1}{q+\tilde{\rho}}-\frac{1}{\tilde{\rho}}\Big)\Delta u +(\mu+\lambda) \Big(\frac{1}{q+\tilde{\rho}}-\frac{1}{\tilde{\rho}}\Big)\nabla \mbox{div}u+k(q),
\end{split}
\end{equation*}
and
\begin{equation*}
\begin{split}
 k(q)&=\left\{
 \begin{array}{llll}
 \nabla\left(\ln\rho-\ln\tilde{\rho}-\tilde{\rho}^{-1}\right),\,\,\, \text{if $\gamma=1$},\\[2mm]
 \nabla\left(\frac{\gamma\rho^{\gamma-1}}{\gamma-1}-\frac{\gamma\tilde{\rho}^{\gamma-1}}{\gamma-1}-\gamma\tilde{\rho}^{\gamma-2}q \right),\,\,\, \text{if $\gamma>1$},
 \end{array}
 \right.\\[2mm]
 &=O(1)q^2+O(1)q\nabla q\\[2mm]
 &=\gamma(\gamma-2)\tilde{\rho}^{\gamma-2}q\nabla q+O(q^2)\nabla q+O(q^2)\nabla \tilde{\rho}.
 \end{split}
\end{equation*}

Next, we note some elliptic estimates of the elliptic system of equations for our domain. $(\ref{per1})_{2}$ is regarded as elliptic with
respect to $x$, that is:
\begin{equation*}\label{elliptic6}
\begin{split}
&\mu\Delta u+(\mu+\lambda)\nabla(\mbox{div} u)=\rho u_t+\nabla(\gamma\tilde{\rho}^{\gamma-1} q)-\tilde{\rho}\nabla\phi-q\nabla\tilde{\Phi}-f,\\[2mm]
&u|_{\partial\Omega}=0.
\end{split}
\end{equation*}
Applying the standard elliptic estimates and the smoothness of $\tilde{\rho}$, $\tilde{\Phi}$ on ${\Omega}$, we have
\begin{lem}\label{lemell}
For $k=0,1$, it holds
\begin{equation*}
\begin{split}
\|D^{k+2}u\|^2
\leq  C\Big\{\|u_t\|_{k}^2+\|q u_t\|^2_{{k}}+\|\nabla q\|_{k}^2+\|q\|_{k}^2+\|\nabla\phi\|_{k}^2+\|f\|_{k}^2\Big\}.
\end{split}
\end{equation*}
\end{lem}
Finally,  an estimate about the stokes equations is given as follows:
\begin{lem} \label{lemmstoke}
For $k=0,1,2$, it holds
\begin{equation}\label{1}
\|D^{k+2}u\|^2+\|D^{k+1}(\gamma\tilde\rho^{\gamma-2}q-\phi)\|^2\leq C\Big\{\left\|\frac{dq}{dt}\right\|_{k+1}^2+\|u\|_{k+1}^2+\|g^0\|_{k+1}^2+\|u_t\|_{k}^2+\|g\|_{k}^2\Big\}.
\end{equation}
\end{lem}
\begin{proof}
Set $U=\tilde{\rho}^{-1}u$, then
 $(\ref{per3-1})$ and $(\ref{per3-2})$ can be rewritten  as:
\begin{equation*}
\begin{split}
\mbox{div} U=&\tilde{\rho}^{-2}\left(-\frac{dq}{dt}-2u\cdot\nabla\tilde{\rho}+g^0\right),\\[2mm]
-\mu\Delta U
+\nabla(\gamma\tilde{\rho}^{\gamma-2}q-\phi)=&-u_t+(\mu+\lambda)\tilde{\rho}^{-1}\nabla \left(\tilde{\rho}^{-1}\left(g^0-\frac{dq}{dt}-\nabla\tilde{\rho}\cdot u\right)\right)+g\\[2mm]
&+\mu\left(2\nabla\tilde\rho^{-1}\cdot\nabla u+\Delta \tilde\rho^{-1} u\right),\\[3mm]
U|_{\partial\Omega}=&0.
\end{split}
\end{equation*}
By the standard estimates of stoke equations, we have
\begin{equation*}
\|D^{k+2}U\|^2+\|D^{k+1}(\gamma\tilde\rho^{\gamma-2}q-\phi)\|^2\leq C\Big\{\left\|\frac{dq}{dt}\right\|_{k+1}^2+\|u\|_{k+1}^2+\|g^0\|_{k+1}^2+\|u_t\|_{k}^2+\|g\|_{k}^2\Big\},
\end{equation*}
which implies (\ref{1}) due to the smoothness of $\tilde{\rho}$.
\end{proof}

\section{Proof of the main result}
Theorem \ref{thm1} will be proved in this section. The local-in-time well-posedness in the smooth norm is quite standard, following the arguments in \cite{matsumura1982}, therefore to prove Theorem \ref{thm1}, it suffices to prove the following {\it a priori} estimates.  For clarity, we introduce
\begin{equation*}
\mathcal{E}(t)=\|(q,u,\nabla\phi)\|_3^2+\|q_t\|_2^2+\|u_t\|_1^2,
\end{equation*}
and
\begin{equation*}
\mathcal{D}(t)=\|(q,\nabla\phi)\|_3^2+\|u\|_4^2+\|q_t\|_2^2+\|u_t\|_2^2.
\end{equation*}
 \begin{prop}\label{prop} ({\bf a priori estimates})
	Let $(q,u,\phi)$ be a solution to the initial boundary value problem (\ref{per1}) in time interval $t\in[0,T]$. Then there exists  positive constants $C$, $\delta$ and $\sigma$ which are independent of $t$,  such that if
	\begin{equation*}
	\sup_{0\leq t\leq T}\mathcal{E}(t)\leq \delta^2,
	\end{equation*}
	then there holds, for any $t\in[0,T]$,
	\begin{equation*}
	\mathcal{E}(t)\leq C\mathcal{E}(0)e^{-\sigma t}.
	\end{equation*}
\end{prop}

We will prove Proposition \ref{prop} in the following Lemmas. To begin with, we have the following basic energy estimate which is quite standard.
\begin{lem}\label{lemma1}
Suppose that the conditions in Proposition \ref{prop} hold, there is a positive constant $C$ independent of $t$, such that
\begin{equation}\label{lem1}
\begin{split}
\frac{1}{2}\frac{d}{dt}\int_{\Omega}\Big(\rho|u|^2+\gamma\tilde{\rho}^{\gamma-2}q^2+|\nabla\phi|^2\Big)dx
+C\left\{\|\nabla u\|^2+\left\|\frac{dq}{dt}\right\|^2 \right\}
\leq C\delta\mathcal{D}(t).
\end{split}
\end{equation}
\end{lem}

\begin{proof}
Rewrite the momentum equation $(\ref{per1})_2$ as
\begin{equation}\label{new1}
\rho\left(u_{t}+u\cdot\nabla u\right)
-\mu\Delta u-(\mu+\lambda)\nabla \mbox{div}u+\nabla(\gamma\tilde{\rho}^{\gamma-1} q)-\tilde{\rho}\nabla\phi-q\nabla\tilde{\Phi}
=q\nabla\phi-\nabla h(q).
\end{equation}
Multiplying $(\ref{new1})$ and $(\ref{per1})_1$  by $u$  and $\gamma \tilde{\rho}^{\gamma-2} q$ respectively, summing up them and integrating by parts with $u|_{\partial\Omega}=0$, we obtain
\begin{equation}\label{lem1-1}
\begin{split}
\frac{1}{2}&\frac{d}{dt}\int_{\Omega}\Big(\rho|u|^2 +\gamma \tilde{\rho}^{\gamma-2}q^2\Big)dx+\mu \int_{\Omega}|\nabla u|^2dx
+(\mu+\lambda)\int_{\Omega}|\mbox{div} u|^2dx \\[2mm]
&+\int_{\Omega}\gamma \tilde{\rho}^{\gamma-2}qu\cdot\nabla\tilde{\rho}dx
-\int_{\Omega}q\nabla\tilde{\Phi}\cdot u dx-\int_{\Omega}\tilde{\rho}\nabla\phi\cdot u dx
\\[2mm]
=&\int_{\Omega}\Big\{\Big(q\nabla\phi{\color{red}-}\nabla h(q)\Big)\cdot u+\gamma\tilde{\rho}^{\gamma-2}qf^{0}\Big\}dx.
\end{split}
\end{equation}
Noting that the first two terms on the second row of (\ref{lem1-1}) can cancel each other because of equation (\ref{steady}). Moreover,
integrating by parts, using equations $(\ref{per1})_1$ and  $(\ref{per1})_3$, we have
\begin{equation*}
\begin{split}
-\int_{\Omega}\tilde{\rho}\nabla\phi\cdot u dx&= \int_{\Omega}\phi \mbox{div}(\tilde{\rho}u) dx
=-\int_{\Omega}\phi q_tdx+ \int_{\Omega}f^0\phi dx\\[2mm]
&=-\int_{\Omega}\phi \Delta\phi_tdx+ \int_{\Omega}f^0\phi dx\\[2mm]
&=\frac{1}{2}\frac{d}{dt}\int_{\Omega}|\nabla\phi|^2dx+ \int_{\Omega}f^0\phi dx.
\end{split}
\end{equation*}
Putting the above equation into (\ref{lem1-1}) yields the basic energy estimate
\begin{equation}\label{lem1-2}
\begin{split}
&\frac{1}{2}\frac{d}{dt}\int_{\Omega}\Big(\rho|u|^2+\gamma\tilde{\rho}^{\gamma-2}q^2+|\nabla\phi|^2\Big)dx
+\int_{\Omega}\left(\mu|\nabla u|^2+(\mu+\lambda)|\mbox{div} u|^2\right) dx=\int_{\Omega}A_0(x,t)dx,
\end{split}
\end{equation}
where
\begin{equation*}\label{A0}
A_0(x,t)=(q\nabla\phi-\nabla h(q))\cdot u+\left(\gamma\tilde{\rho}^{\gamma-2}q-\phi\right)f^{0}.
\end{equation*}
And it is clear that
\begin{equation}\label{lem1-3}
\int_{\Omega}A_0(x,t)dx\leq C\delta\mathcal{D}(t).
\end{equation}
Moreover,
\begin{equation*}
\frac{dq}{dt}=-\rho \mbox{div}u-u\cdot\nabla\tilde\rho,
\end{equation*}
and the  Poincar$\acute{e}$'s inequlity for $u$ give
\begin{equation*}
\left\|\frac{dq}{dt}\right\|^2\leq C(\|\mbox{div}u\|^2+\|\nabla u\|^2),
\end{equation*}
 which together with (\ref{lem1-2}) and (\ref{lem1-3}) implies  the desired result (\ref{lem1}).
\end{proof}

The next lemma is $L^2$-estimates for $t$-derivatives of $(q,u,\nabla\phi)$.
\begin{lem}\label{lemma2}
Under the assumptions in Proposition \ref{prop}, there exists a constant $C>0$ independent of $t$ such that
\begin{equation}\label{lem2}
\begin{split}
\frac{1}{2}\frac{d}{dt}\int_{\Omega}\Big(\rho|u_t|^2+\gamma\tilde{\rho}^{\gamma-2}q_t^2+|\nabla\phi_t|^2\Big)dx
+C\left\{\|\nabla u_t\|^2+\left\|\left(\frac{dq}{dt}\right)_t\right\|^2 \right\}
\leq C\delta\mathcal{D}(t).
\end{split}
\end{equation}
\end{lem}

\begin{proof}
Notice that differentiation of the system (\ref{per1}) with respect to $t$ will keep the boundary
conditions $(\ref{per1})_4$ and $(\ref{per1})_5$. Estimating the integral for
\begin{equation*}
 \int_{\Omega}\Big\{\partial_{t}(\ref{per1})_1\gamma\tilde{\rho}^{\gamma-2}q_t
 +\partial_{t}(\ref{per1})_2\cdot u_t\Big\}dx=0,
 \end{equation*}
  and noting that
  \begin{equation*}
  \int_{\Omega}\partial_{t}(\rho u_t)\cdot u_tdx=\frac{1}{2}\frac{d}{dt}\int_{\Omega}\rho|u_t|^2dx+\frac{1}{2}\int_{\Omega}\rho_t|u_t|^2dx, \end{equation*}
then, using the similar way as in Lemma \ref{lemma1} shows
\begin{equation*}
\begin{split}
&\frac{1}{2}\frac{d}{dt}\int_{\Omega}\Big(\rho|u_t|^2+\gamma\tilde{\rho}^{\gamma-2}q_t^2+|\nabla\phi_t|^2\Big)dx
+C\int_{\Omega}\left(|\nabla u_t|^2+|\mbox{div} u_t|^2 \right)dx\\[2mm]
&=-\frac{1}{2}\int_{\Omega}\rho_t|u_t|^2dx+\int_{\Omega}A_1(x,t)dx
\end{split}
\end{equation*}
where
\begin{equation*}
\int_{\Omega}A_1(x,t)dx=\int_{\Omega}\left\{f_t\cdot u_t+\gamma\tilde{\rho}^{\gamma-2}f^{0}_tq_t-f_t^{0}\phi_t\right\}dx\leq C\delta\mathcal{D}(t),
\end{equation*}
which gives the desired result (\ref{lemma2}) by using the following estimate
\begin{equation*}
\left\|\left(\frac{dq}{dt}\right)_{t}\right\|^2\leq C\left(\|\mbox{div}u_t\|^2+\|\nabla u_{t}\|^2+\delta\mathcal{D}(t)\right).
\end{equation*}
\end{proof}

Now, we would like to give the estimate  spatial derivatives of $(q,u,\nabla\phi)$.
\begin{lem}\label{lemma3}
Suppose that the conditions in Proposition \ref{prop} hold, there is a positive constant $C$ independent of $t$, such that
\begin{equation}\label{lem3}
\begin{split}
&\frac{1}{2}\frac{d}{dt}\int_{\Omega}\Big(\mu|\nabla u|^2+(\mu+\lambda)|\mbox{div} u|^2\Big)dx
-\frac{d}{dt}\int_{\Omega}\Big(\tilde{\rho}\nabla\phi\cdot u +q u\cdot\nabla\tilde{\Phi}
+\gamma\tilde{\rho}^{\gamma-1}q\mbox{div}u\Big)dx\\[2mm]
&+C\left(\|q_t\|^2+\|u_t\|^2\right)
\leq C\|\nabla u\|^2+C\delta\mathcal{D}(t).
\end{split}
\end{equation}
\end{lem}

\begin{proof}
Testing $(\ref{per1})_1$ and $(\ref{per1})_2$ with $\gamma \tilde{\rho}^{\gamma-2} q_t$ and $u_t$ respectively, summing up them, we obtain
\begin{equation}\label{3-1}
\begin{split}
\frac{1}{2}&\frac{d}{dt}\int_{\Omega}\Big(\mu|\nabla u|^2+(\mu+\lambda)|\mbox{div} u|^2\Big)dx
+\int_{\Omega}\gamma \tilde{\rho}^{\gamma-2}q_t^2dx+\int_{\Omega}{\rho}|u_t|^2dx\\[2mm]
&+\int_{\Omega}\gamma\tilde{\rho}^{\gamma-1}\mbox{div} u q_t dx
+\int_{\Omega}\nabla(\gamma\tilde{\rho}^{\gamma-1}q)\cdot u_tdx\\[2mm]
&+\int_{\Omega}\gamma\tilde{\rho}^{\gamma-2}u\cdot\nabla\tilde{\rho}q_tdx
-\int_{\Omega}q\nabla\tilde{\Phi}\cdot u_tdx-\int_{\Omega}\tilde{\rho}\nabla\phi\cdot u_tdx\\[2mm]
=&\int_{\Omega}\left\{\gamma{\tilde{\rho}}^{\gamma-2}f^{0}q_t+f\cdot u_t\right\}dx.
\end{split}
\end{equation}
Integrating  by parts gives
\begin{equation}
\begin{split}
\int_{\Omega}\nabla(\gamma\tilde{\rho}^{\gamma-1}q)\cdot u_tdx
&=-\int_{\Omega}\gamma\tilde{\rho}^{\gamma-1}q\mbox{div}u_{t}dx\\[2mm]
&=-\frac{d}{dt}\int_{\Omega}\gamma\tilde{\rho}^{\gamma-1}q\mbox{div}udx+\int_{\Omega}\gamma\tilde{\rho}^{\gamma-1}q_t\mbox{div}udx.
\end{split}
\end{equation}
By equation (\ref{steady}), it indicates
\begin{equation*}
\begin{split}
\int_{\Omega}\gamma\tilde{\rho}^{\gamma-2}u\cdot\nabla\tilde{\rho}q_tdx-\int_{\Omega}q\nabla\tilde{\Phi}\cdot u_tdx
&=\int_{\Omega}q_{t}u\cdot\nabla\tilde{\Phi}dx-\int_{\Omega}q\nabla\tilde{\Phi}\cdot u_tdx\\[2mm]
&=-\frac{d}{dt}\int_{\Omega}q\nabla\tilde{\Phi}\cdot udx+2\int_{\Omega}q_{t}u\cdot\nabla\tilde{\Phi}dx.
\end{split}
\end{equation*}
In view of  $(\ref{per1})_{1}$ and $(\ref{per1})_{3}$, one has
\begin{equation*}
\begin{split}
-\int_{\Omega}\tilde{\rho}\nabla\phi\cdot u_{t}dx
&=-\frac{d}{dt}\int_{\Omega}\tilde{\rho}\nabla\phi\cdot udx+\int_{\Omega}\tilde{\rho}\nabla\phi_t\cdot u dx\\[2mm]
&=-\frac{d}{dt}\int_{\Omega}\tilde{\rho}\nabla\phi\cdot udx+\int_{\Omega}\nabla\phi_tq_t dx-\int_{\Omega}f^{0}\phi_{t}dx\\
&=-\frac{d}{dt}\int_{\Omega}\tilde{\rho}\nabla\phi\cdot udx-\int_{\Omega}|\nabla\phi_t|^2dx-\int_{\Omega}f^{0}\phi_{t}dx.
\end{split}
\end{equation*}
Putting all the above identities into (\ref{3-1}), one obtains
\begin{equation}\label{lem3-1}
\begin{split}
\frac{1}{2}&\frac{d}{dt}\int_{\Omega}\Big(\mu|\nabla u|^2+(\mu+\lambda)|\mbox{div} u|^2\Big)dx
-\frac{d}{dt}\int_{\Omega}\tilde{\rho}\nabla\phi\cdot u dx-\frac{d}{dt}\int_{\Omega}q\nabla\tilde{\Phi}\cdot u dx\\
&-\frac{d}{dt}\int_{\Omega}\gamma\tilde{\rho}^{\gamma-1}q\mbox{div}udx+\int_{\Omega}\gamma \tilde{\rho}^{\gamma-2}q_t^2dx+\int_{\Omega}\tilde{\rho}|u_t|^2dx\\
=&\int_{\Omega}\left(|\nabla\phi_t|^2-2q_t\nabla\tilde{\Phi}\cdot u-2\gamma\tilde{\rho}^{\gamma-1}q_{t}\mbox{div} u \right)dx+\int_{\Omega}\left(\gamma \tilde{\rho}^{\gamma-2}f^{0}q_t+f\cdot u_{t}+f^{0}\phi_t\right)dx.
\end{split}
\end{equation}
Now, the terms on the right-hand side of (\ref{lem3-1}) will be estimated. By using the equation  $\Delta\phi_t=-\mbox{div}(\rho u)$ and Poincar$\acute{e}$'s inequality for $u$  with $u|_{\partial\Omega}=0$, we infer the following important estimate:
\begin{equation}\label{imp1}
\|\nabla\phi_t\|^2\leq C\|\rho u\|^2\leq C\|u\|^2\leq C\|\nabla u\|^2.
\end{equation}
Utilizing Cauchy's inequality and Poincar$\acute{e}$'s inequality yields
\begin{equation*}
\begin{split}
\int_{\Omega}q_t\nabla\tilde{\Phi}\cdot udx&\leq \varepsilon \int_{\Omega}q_t^2dx+C\int_{\Omega}|u|^2dx\\[2mm]
&\leq  \varepsilon \|q_t\|^2+C\|\nabla u\|^2
\end{split}
\end{equation*}
and
\begin{equation*}
\int_{\Omega}\gamma\tilde{\rho}^{\gamma-1}q_{t}\mbox{div} u dx\leq \varepsilon \|q_t\|^2+C\|\nabla u\|^2.
\end{equation*}
Finally, the following nonlinear term is controlled by
\begin{equation*}
\int_{\Omega}\Big\{\gamma \tilde{\rho}^{\gamma-2}f^{0}q_t+f\cdot u_{t}+f^{0}\phi_t\Big\}dx\leq C\delta\mathcal{D}^{2}(t).
\end{equation*}
Therefore, the proof of Lemma \ref{lemma3} is completed.
\end{proof}

\begin{lem}\label{lemma4}
Under the assumptions in Proposition \ref{prop}, there exists a constant $C>0$ independent of $t$ such that
\begin{equation}\label{lem4}
\begin{split}
&\frac{1}{2}\frac{d}{dt}\int_{\Omega}\Big(\mu|\nabla u_t|^2+(\mu+\lambda)|\mbox{div}u_t|^2\Big)dx
-\frac{d}{dt}\int_{\Omega}\gamma\tilde{\rho}^{\gamma-1}q_t\mbox{div}u_tdx+C\left(\|q_{tt}\|^2+\|u_{tt}\|^2\right) \\[1mm]
&\leq C\Big(\|\nabla u_t\|^2+\|\nabla u\|^2+\|u_t\|^2+\|q_t\|^2\Big)+C\delta\mathcal{D}(t).
\end{split}
\end{equation}

\end{lem}

\begin{proof}
Taking $\partial_{t}$ to $(\ref{per1})_1$,  $(\ref{per1})_2$, multiplying the resulting identities by $\gamma\tilde{\rho}^{\gamma-2}q_{tt}$ and $u_{tt}$ respectively, and summing up them, the following equation is arrived:
\begin{equation}\label{lem4-1}
\begin{split}
\frac{1}{2}&\frac{d}{dt}\int_{\Omega}\Big(\mu|\nabla u_t|^2+(\mu+\lambda)|div u_t|^2\Big)dx
+\int_{\Omega}\gamma \tilde{\rho}^{\gamma-2}q_{tt}^2dx+\int_{\Omega}{\rho}|u_{tt}|^2dx+\int_{\Omega}\rho_tu_t\cdot u_{tt}dx\\[2mm]
&+\int_{\Omega}\gamma\tilde{\rho}^{\gamma-1}\mbox{div} u_{t}q_{tt}dx
+\int_{\Omega}\nabla(\gamma\tilde{\rho}^{\gamma-1}q_t)\cdot u_{tt}dx\\[2mm]
&+\int_{\Omega}\gamma\tilde{\rho}^{\gamma-2}u_t\cdot\nabla\tilde{\rho}q_{tt}dx-\int_{\Omega}q_t\nabla\tilde{\Phi}\cdot u_{tt}dx
-\int_{\Omega}\tilde{\rho}\nabla\phi_t\cdot u_{tt}dx\\[2mm]
=&\int_{\Omega}\left(\gamma{\tilde{\rho}}^{\gamma-2}f_t^{0}q_{tt}+f_t\cdot u_{tt}\right)dx.
\end{split}
\end{equation}
The second row on the left-hand side of (\ref{lem4-1}) becomes
\begin{equation}\label{lem4-2}
\begin{split}
\int_{\Omega}\left[\gamma\tilde{\rho}^{\gamma-1}\mbox{div} u_{t}q_{tt}
+\nabla(\gamma\tilde{\rho}^{\gamma-1}q_t)\cdot u_{tt}\right]dx
&=\int_{\Omega}\gamma\tilde{\rho}^{\gamma-1}\mbox{div} u_{t}q_{tt}dx
-\int_{\Omega}\gamma\tilde{\rho}^{\gamma-1}q_t \mbox{div}u_{tt}dx\\[2mm]
&=-\frac{d}{dt}\int_{\Omega}\gamma\tilde{\rho}^{\gamma-1}\mbox{div} u_{t}q_{t}dx+2 \int_{\Omega}\gamma\tilde{\rho}^{\gamma-1}\mbox{div} u_{t}q_{tt}dx\\[2mm]
&\geq -\frac{d}{dt}\int_{\Omega}\gamma\tilde{\rho}^{\gamma-1}\mbox{div} u_{t}q_{t}dx- \varepsilon \|q_{tt}\|^2-C\|\nabla u_{t}\|^2.
\end{split}
\end{equation}
Different from that in Lemma \ref{lemma3}, the third row on the left-hand side of (\ref{lem4-1}) is bounded by
\begin{equation}\label{lem4-3}
\begin{split}
&\int_{\Omega}\gamma\tilde{\rho}^{\gamma-2}u_t\cdot\nabla\tilde{\rho}q_{tt}dx-\int_{\Omega}q_t\nabla\tilde{\Phi}\cdot u_{tt}dx-\int_{\Omega}\tilde{\rho}\nabla\phi_t\cdot u_{tt}dx\\[2mm]
&\leq \varepsilon \Big(\|q_{tt}\|^2+ \|u_{tt}\|^2\Big)+C\Big(\|u_{t}\|^2+\|\nabla\phi_t\|^2+\|q_t\|^2\Big)\\[2mm]
&\leq \varepsilon \Big(\|q_{tt}\|^2+ \|u_{tt}\|^2\Big)+C\Big(\|u_{t}\|^2+\|\nabla u\|^2+\|q_t\|^2\Big),
\end{split}
\end{equation}
where we have used Cauchy's inequality and the fact (\ref{imp1}). Consequently, combining (\ref{lem4-1})-(\ref{lem4-3}) yields the desired result (\ref{lem4}).
\end{proof}

We will use the momentum equation in view of (\ref{per3-2}) to get the $H^1$-norm of $q$ in the following lemma.
\begin{lem}\label{lemma5}
Under the conditions in Proposition \ref{prop},  it holds that
\begin{equation}\label{lem5}
\|\nabla\phi\|^2+\|q\|^2+\|\nabla q\|^2
\leq C\Big(\|u_t\|^2+\|D^2u\|^2\Big)+C\delta\mathcal{D}(t).
\end{equation}
\end{lem}

\begin{proof}
The momentum equation (\ref{per3-2}) could be rewritten as
\begin{equation}\label{per2}
-\nabla\phi+\nabla(\gamma\tilde{\rho}^{\gamma-2}q)=-u_t+\frac{\mu}{q+\tilde{\rho}}\Delta u
+\frac{\mu+\lambda}{q+\tilde{\rho}}\nabla \mbox{div}u-(u\cdot\nabla)u+k(q).
\end{equation}
Taking the inner product of  (\ref{per2}) with  $-\nabla\phi$, and  using the boudary condition $\nabla\phi\cdot\nu|_{\partial\Omega}=0 $,
the left-hand side becomes
\begin{equation*}
\int_{\Omega}|\nabla\phi|^2dx+\gamma\int_{\Omega}\tilde{\rho}^{\gamma-2}q^2dx
\end{equation*}
due to
\begin{equation*}
-\int_{\Omega}\nabla(\gamma\tilde{\rho}^{\gamma-2}q)\cdot\nabla\phi dx=\int_{\Omega}\gamma\tilde{\rho}^{\gamma-2}q\Delta\phi dx=\int_{\Omega}\gamma\tilde{\rho}^{\gamma-2}q^2dx.
\end{equation*}
Therefore, we obtain
\begin{equation*}
\begin{split}
\frac{1}{2}\int_{\Omega}|\nabla\phi|^2dx+\gamma\int_{\Omega}\tilde{\rho}^{\gamma-2}q^2dx&\leq C\Big(\|u_t\|^2+\|D^2u\|^2\Big)
+C\Big(\|(u\cdot\nabla)u\|^2+\|k(q)\|^2\Big)\\
&\leq  C\Big(\|u_t\|^2+\|D^2u\|^2\Big)+C\delta\mathcal{D}(t)
\end{split}
\end{equation*}
by using (\ref{per2}) and Cauchy's inequality. Furthermore,
\begin{equation*}
\|\nabla q\|^2\leq C\Big(\|u_t\|^2+\|D^2u\|^2\Big)+C\delta\mathcal{D}(t),
\end{equation*}
by means of
\begin{equation*}
\nabla(\tilde{\rho}^{\gamma-2}q)=(\gamma-2)\tilde{\rho}^{\gamma-3}\nabla\tilde{\rho}q+\tilde{\rho}^{\gamma-2}\nabla q,
\end{equation*}
and
\begin{equation*}
\|\nabla(\tilde{\rho}^{\gamma-2}q)\|^2\leq C\left(\|\nabla\phi\|^2+\|u_t\|^2+\|D^2u\|^2\right)+C\delta\mathcal{D}(t).
\end{equation*}
Hence, we have finished the proof of this lemma.
\end{proof}

From now on, we shall separate the estimates into that away from the boundary and that near the boundary. Let $\chi_0(x)$ be any fixed $C^{\infty}(\Omega)$ cut-off function such that $supp\chi_{0}\equiv K\subset\subset\Omega$, and $\chi_0\equiv1$ in $K_1\subset\subset K$. With the help of $\chi_0(x)$, we have the estimates in the interior domain.

\begin{lem}\label{lemma6}
Assume that the conditions in Proposition \ref{prop} hold, then  for any positive $\epsilon$, it holds that
\begin{equation}\label{lem6-2}
\begin{split}
 \frac{\gamma}{2}&\frac{d}{dt}\int_{\Omega}\chi_0^2\tilde{\rho}^{\gamma-4}|D q|^2dx
-\frac{d}{dt}\Big\{\int_{\Omega}\chi_0^2\tilde{\rho}^{-2} \nabla q \cdot \nabla\phi dx
-\int_{\Omega}\chi_0^2\tilde{\rho}^{-2}\nabla(\gamma\tilde{\rho}^{\gamma-2})\cdot\nabla q q dx\Big\}\\[2mm]
&+C\left\{\|\chi_0D(\gamma\tilde{\rho}^{\gamma-2} q-\phi)\|^2
+\|\chi_0D^2u\|^2+\left\|\chi_0D\frac{dq}{dt}\right\|^2\right\}\\[2mm]
\leq&\epsilon \|D q\|^2+C\Big\{\|q_t\|^2+\|D u\|^2+\|u_t\|^2\Big\}+C\delta\mathcal{D}(t),
\end{split}
\end{equation}
and the estimates of  derivatives of high order:
 \begin{equation}\label{lem6-3}
\begin{split}
 \frac{\gamma}{2}&\frac{d}{dt}\int_{\Omega}\chi_0^2\tilde{\rho}^{\gamma-4}|D^{2} q|^2dx
-\frac{d}{dt}\Big\{\int_{\Omega}\chi_0^2\tilde{\rho}^{-2} D_{ij}q  D_{ij} \phi dx\\[2mm]
&-2\int_{\Omega}\chi_0^2\tilde{\rho}^{-2}D_{i}(\gamma\tilde{\rho}^{\gamma-2})\cdot D_{ij}q D_{j}q dx
-\int_{\Omega}\chi_0^2\tilde{\rho}^{-2}D_{ij}(\gamma\tilde{\rho}^{\gamma-2})D_{ij}q q dx\Big\}\\[2mm]
&+C\left\{\|\chi_0D^2(\gamma\tilde{\rho}^{\gamma-2} q-\phi)\|^2
+\|\chi_0D^3u\|^2+\left\|\chi_0D^2\frac{dq}{dt}\right\|^2
\right\}\\[2mm]
\leq&\epsilon \|D^2 q\|^2+C\Big\{\|q_t\|_1^2+\|D^{2} u\|^2+\|Du_t\|^2\Big\}+C\delta\mathcal{D}(t),
\end{split}
\end{equation}
and
\begin{equation}\label{lem6-4}
\begin{split}
\frac{\gamma}{2}&\frac{d}{dt}\int_{\Omega}\chi_0^2\tilde{\rho}^{\gamma-4}|D^{3} q|^2dx
-\frac{d}{dt}\Big\{\int_{\Omega}\chi_0^2\tilde{\rho}^{-2} D_{ijk}  q  D_{ijk} \phi dx
-3\int_{\Omega}\chi_0^2\tilde{\rho}^{-2}D_{i}(\gamma\tilde{\rho}^{\gamma-2})\cdot D_{ijk}q D_{jk}q dx\\[2mm]
&-3\int_{\Omega}\chi_0^2\tilde{\rho}^{-2}D_{ij}(\gamma\tilde{\rho}^{\gamma-2})D_{ijk}q D_{k}q dx
-\int_{\Omega}\chi_0^2\tilde{\rho}^{-2}D_{ijk}(\gamma\tilde{\rho}^{\gamma-2})D_{ijk}q q dx\Big\}\\[2mm]
&+C\left\{\|\chi_0D^3(\gamma\tilde{\rho}^{\gamma-2} q-\phi)\|^2
+\|\chi_0D^4u\|^2+\left\|\chi_0D^{3}\frac{dq}{dt}\right\|^2
\right\}\\[2mm]
\leq&\epsilon \|D^3 q\|^2+C\Big\{\|q_t\|_2^2+\|D^{3} u\|^2+\|D^{2}u_t\|^2\Big\}+C\delta\mathcal{D}(t).
\end{split}
\end{equation}

\end{lem}

\begin{proof}
 Testing $\nabla (\ref{per1})_1$, (\ref{per3-2}) with  $(2\mu+\lambda)\chi_0^2\tilde{\rho}^{-2}\nabla(\gamma\tilde{\rho}^{\gamma-2} q-\phi)$ and $\chi_0^2\nabla(\gamma\tilde{\rho}^{\gamma-2} q-\phi)$ respectively, then integrating over $\Omega$, one obtains
\begin{equation}\label{lem6-1-1}
\begin{split}
(2\mu&+\lambda)\int_{\Omega}\chi_0^2\tilde{\rho}^{-2}\nabla q_{t}\cdot\nabla(\gamma\tilde{\rho}^{\gamma-2} q-\phi)dx
 +\int_{\Omega}\chi_0^2|\nabla(\gamma\tilde{\rho}^{\gamma-2} q-\phi)|^2dx\\[2mm]
=&\int_{\Omega}\mu\chi_0^2\tilde{\rho}^{-1}\Big(\Delta u-\nabla(\mbox{div}u)\Big)\cdot
 \nabla(\gamma\tilde{\rho}^{\gamma-2} q-\phi)dx
 -\int_{\Omega}\chi_0^2 u_{t}\cdot\nabla(\gamma\tilde{\rho}^{\gamma-2} q-\phi)dx\\[2mm]
 &-(2\mu+\lambda)\int_{\Omega}\chi_0^2\tilde{\rho}^{-2}\Big(\nabla\tilde{\rho}\mbox{div}u+\nabla(u\cdot\nabla\tilde\rho)\Big)
\cdot\nabla(\gamma\tilde{\rho}^{\gamma-2} q-\phi)dx\\[2mm]
 &+\int_{\Omega}\chi_0^2\Big\{(2\mu+\lambda)\tilde{\rho}^{-2}\nabla f^0+g\Big\}\cdot\nabla(\gamma\tilde{\rho}^{\gamma-2} q-\phi)dx.
\end{split}
\end{equation}
The first term on the  left-hand side of (\ref{lem6-1-1}) has the following lower bound:
\begin{equation}\label{first}
\begin{split}
\int_{\Omega}&\chi_0^2\tilde{\rho}^{-2}\nabla q_{t}\cdot\nabla(\gamma\tilde{\rho}^{\gamma-2} q-\phi)dx\\[1mm]
=&\frac{\gamma}{2}\frac{d}{dt}\int_{\Omega}\chi_0^2\tilde{\rho}^{\gamma-4}|\nabla q|^2dx
+\frac{d}{dt}\int_{\Omega}\chi_0^2\tilde{\rho}^{-2}\nabla(\gamma\tilde{\rho}^{\gamma-2})\cdot\nabla q q dx
-\int_{\Omega}\chi_0^2\tilde{\rho}^{-2}\nabla(\gamma\tilde{\rho}^{\gamma-2})\cdot\nabla q q_t dx\\[1mm]
&-\frac{d}{dt}\int_{\Omega}\chi_0^2\tilde{\rho}^{-2} \nabla q \cdot \nabla\phi dx
+\int_{\Omega}\chi_0^2 \tilde{\rho}^{-2}\nabla q \cdot \nabla\phi_t dx\\[2mm]
\geq& \frac{\gamma}{2}\frac{d}{dt}\int_{\Omega}\chi_0^2\tilde{\rho}^{\gamma-4}|\nabla q|^2dx
+\frac{d}{dt}\int_{\Omega}\chi_0^2\tilde{\rho}^{-2}\nabla(\gamma\tilde{\rho}^{\gamma-2})\cdot\nabla q q dx\\[1mm]
&-\frac{d}{dt}\int_{\Omega}\chi_0^2\tilde{\rho}^{-2} \nabla q \cdot \nabla\phi dx
-\varepsilon\|\nabla q\|^2-C\Big(\|q_t\|+\|\nabla u\|^2\Big),
\end{split}
\end{equation}
where we have used the estimate $\|\nabla\phi_t\|^2\leq C\|\nabla u\|^2$ in the last step. The first term on the  right-hand side (\ref{lem6-1-1}) of can be estimated as
\begin{equation*}
\begin{split}
&\int_{\Omega}\mu\chi_0^2\tilde{\rho}^{-1}\Big(\Delta u-\nabla(\mbox{div} u)\Big)\cdot
 \nabla(\gamma\tilde{\rho}^{\gamma-2} q-\phi)dx\\[2mm]
&=-\mu \int_{\Omega}\partial_{x_j}\Big(\chi_0^2\tilde{\rho}^{-1}(\gamma \tilde{\rho}^{\gamma-2}q-\phi)_{x_i}\Big)u^{i}_{x_j} dx
+\mu \int_{\Omega}\partial_{x_j}\Big(\chi_0^2\tilde{\rho}^{-1}(\gamma \tilde{\rho}^{\gamma-2}q-\phi)_{x_i}\Big)u^{j}_{x_i} dx \\[2mm]
&=-\mu \int_{\Omega}\partial_{x_j}(\chi_0^2\tilde\rho^{-1}) (\gamma \tilde{\rho}^{\gamma-2}q-\phi)_{x_i}u^{i}_{x_j} dx
+\mu \int_{\Omega}\partial_{x_j}(\chi_0^2\tilde\rho^{-1})(\gamma \tilde{\rho}^{\gamma-2}q-\phi)_{x_i}u^{j}_{x_i} dx\\[2mm]
&\leq  \frac{1}{4} \int_{\Omega}\chi_0^2|\nabla(\gamma\tilde{\rho}^{\gamma-2} q-\phi)|^2dx+ C\int_{\Omega}|\nabla u|^2dx.
\end{split}
\end{equation*}
 Putting all the above inequalities into (\ref{lem6-1-1}), it implies
\begin{equation}\label{lem6-1}
\begin{split}
 \frac{\gamma}{2}&\frac{d}{dt}\int_{\Omega}\chi_0^2\tilde{\rho}^{\gamma-4}|\nabla q|^2dx
+\gamma(\gamma-2)\frac{d}{dt}\int_{\Omega}\chi_0^2\tilde{\rho}^{\gamma-5}q\nabla q\cdot\nabla\tilde{\rho} dx\\[2mm]
&-\frac{d}{dt}\int_{\Omega}\chi_0^2\tilde{\rho}^{-2} \nabla q \cdot \nabla\phi dx
+\frac{1}{4}\int_{\Omega}\chi_0^2|\nabla(\gamma\tilde{\rho}^{\gamma-2} q-\phi)|^2dx\\[2mm]
\leq&\epsilon \|\nabla q\|^2+C\left(\|q_t\|^2+\|\nabla u\|^2+\|u_t\|^2\right)dx+C\delta\mathcal{D}(t)
\end{split}
\end{equation}
after using Cauchy's inequality.

Next, we deal with  $\nabla(\ref{per1})_1\cdot\chi_0^2\nabla(\gamma\tilde{\rho}^{\gamma-2} q-\phi)+ \nabla(\ref{per3-2})^{i}\cdot \chi_0^2\nabla(\tilde{\rho} u^{i})$, and integrate the yeilding result over $\Omega$ to have
\begin{equation}\label{6step4-3}
\begin{split}
\int_{\Omega}&\chi_0^2  \nabla q_t\cdot \nabla(\gamma\tilde{\rho}^{\gamma-2} q-\phi)dx
+\int_{\Omega}\chi_0^2  \nabla u^{i}_t \cdot\nabla (\tilde{\rho}u^{i})dx\\[2mm]
&+ \int_{\Omega} \Big\{\chi_0^2 \nabla \mbox{div}(\tilde{\rho}u)\cdot \nabla(\gamma\tilde{\rho}^{\gamma-2} q-\phi)
+ \chi_0^2  \nabla\partial_{x_i}(\gamma\tilde\rho^{\gamma-2} q-\phi)\cdot\nabla(\tilde\rho u^{i})\Big\}dx\\[2mm]
&- \int_{\Omega} \Big\{\mu\chi_0^2 \nabla(\tilde{\rho}^{-1}\Delta u^{i})\cdot \nabla(\tilde\rho u^{i})
+ (\mu+\lambda)\chi_0^2  \nabla(\tilde\rho^{-1}\partial_{x_i}(\mbox{div}u))\cdot\nabla(\tilde\rho u^{i})\Big\}dx
\\[2mm]
=&\int_{\Omega}\chi_0^2\nabla f^{0}\cdot\nabla(\gamma\tilde\rho^{\gamma-2}q-\phi)dx+\int_{\Omega}\chi_0^2\nabla g^{i}\cdot\nabla (\tilde\rho u^{i}) dx.
\end{split}
\end{equation}
By the same argument as (\ref{first}), it showes us that
\begin{equation*}
\begin{split}
\int_{\Omega}&\chi_0^2\nabla q_{t}\cdot\nabla(\gamma\tilde{\rho}^{\gamma-2} q-\phi)dx\\[1mm]
\geq &\frac{\gamma}{2}\frac{d}{dt}\int_{\Omega}\chi_0^2\tilde{\rho}^{\gamma-2}|\nabla q|^2dx
+\frac{d}{dt}\int_{\Omega}\chi_0^2\nabla(\gamma\tilde{\rho}^{\gamma-2})\cdot\nabla q q dx\\[2mm]
&-\frac{d}{dt}\int_{\Omega}\chi_0^2 \nabla q \cdot \nabla\phi dx
-\varepsilon\|D q\|^2-C\Big(\|q_t\|^2+\|D u\|^2\Big).
\end{split}
\end{equation*}
 It is easy to see that
\begin{equation*}
\begin{split}
\int_{\Omega} \chi_0^2  \nabla u^{i}_t \cdot\nabla (\tilde{\rho}u^{i}) dx
&=\frac{1}{2}\frac{d}{dt}\int_{\Omega}\chi_0^2|\nabla u|^2dx+\int_{\Omega}\chi_0^2\nabla u^{i}_{t}\cdot\nabla\tilde\rho u^{i}dx\\[2mm]
&=\frac{1}{2}\frac{d}{dt}\int_{\Omega}\chi_0^2|\nabla u|^2dx-\int_{\Omega}u^{i}_{t}\partial_{x_j}\left(\chi_0^2\partial_{x_j}\tilde\rho u^{i}\right)dx\\[2mm]
&\geq\frac{1}{2}\frac{d}{dt}\int_{\Omega}\chi_0^2|\nabla u|^2dx-C(\|D u\|^2+\|u_t\|^2).
\end{split}
\end{equation*}
by Cauchy's inequality and Poincar$\acute{e}$'s inequality. While for the following terms of (\ref{6step4-3}), one has
\begin{equation*}
\begin{split}
& \int_{\Omega} \Big\{\chi_0^2 \nabla \mbox{div}(\tilde{\rho}u)\cdot \nabla(\gamma\tilde{\rho}^{\gamma-2} q-\phi)
+ \chi_0^2  \nabla\partial_{x_i}(\gamma\tilde\rho^{\gamma-2} q-\phi)\cdot\nabla(\tilde\rho u^{i})\Big\}dx\\[2mm]
&=- \int_{\Omega} \partial_{x_i}\chi_0^2\nabla(\gamma\tilde\rho^{\gamma-2}q-\phi)\cdot\nabla(\tilde\rho u^{i})dx\\[2mm]
&\geq -\frac{1}{8}\int_{\Omega} \chi_0^2|\nabla(\gamma\tilde\rho^{\gamma-2} q-\phi)|^2dx -C\|D u\|^2
\end{split}
\end{equation*}
due to Cauchy's inequality and Poincar$\acute{e}$'s inequality.  In the meantime, we have
\begin{equation*}
\begin{split}
- \int_{\Omega} \chi_0^2 \partial_{x_j}(\tilde\rho^{-1}\Delta u^{i})\cdot  \partial_{x_j} (\tilde\rho u^{i})dx
\geq \frac{1}{2}\int_{\Omega} \chi_0^2 |D^2u|^2dx-C\|D u\|^2,
\end{split}
\end{equation*}
by utilizing integration by parts, Cuachy's inequality and the elliptic estimate for bounded domain to have $C^{-1}\|D^2u\|^2\leq \|\Delta u\|^2\leq C\|D^2u\|^2$. It is clear that
\begin{equation*}
\begin{split}
-\int_{\Omega} \chi_0^2 \partial_{x_j}(\tilde\rho^{-1}\partial_{x_i}\mbox{div}u)\cdot \partial_{x_j} (\tilde\rho u^{i})dx
&\geq \frac{1}{2}\int_{\Omega} \chi_0^2 |D \mbox{div}u|^2dx-C\|D u\|^2\\[2mm]
&\geq  C\int_{\Omega} \chi_0^2 \left|D\frac{dq}{dt}\right|^2dx-C\|D u\|^2.
\end{split}
\end{equation*}
Putting all the above inequalities into (\ref{6step4-3}), it implies that
\begin{equation*}
\begin{split}
\frac{d}{dt}&\int_{\Omega}\left(\frac{\gamma}{2}\chi_0^2\tilde{\rho}^{\gamma-2}|\nabla q|^2+\frac{1}{2}\chi_0^2|\nabla u|^2
+\gamma(\gamma-2)\chi_0^2\tilde{\rho}^{\gamma-5}q\nabla q\cdot\nabla\tilde{\rho}
-\chi_0^2 \nabla q \cdot \nabla\phi\right) dx\\[2mm]
&+\int_{\Omega} \Big\{\chi_0^2 |D^2u|^2+ \chi_0^2 |D\mbox{div}u|^2\Big\}dx \\[2mm]
\leq& \frac{1}{8}\int_{\Omega}\chi_0^2 |\nabla(\gamma\tilde\rho^{\gamma-2} q-\phi)|^2dx+C\Big(\|q_t\|^2+\|\nabla u\|^2+\|u_t\|^2\Big)+C\delta\mathcal{D}(t),
\end{split}
\end{equation*}
which together with (\ref{lem6-1}) yields (\ref{lem6-2}).

The estimate (\ref{lem6-3}) for second-order derivatives could be obtained similarly, i.e.,  estimating the following two integrals
 \begin{equation}{\label{int1}}
\int_{\Omega}\left\{D_{ij}(\ref{per1})_1 (2\mu+\lambda)\chi_0^2\tilde{\rho}^{-2}D_{ij}(\gamma\tilde{\rho}^{\gamma-2} q-\phi)+ D_{j}(\ref{per3-2})^{i} \chi_0^2D_{ij}(\gamma\tilde{\rho}^{\gamma-2} q-\phi)\right\}dx=0,
\end{equation}
and
\begin{equation}{\label{int2}}
\int_{\Omega}\left\{D_{jk}(\ref{per1})_1 \chi_0^2D_{jk}(\gamma\tilde{\rho}^{\gamma-2} q-\phi)+ D_{jk}(\ref{per3-2})^{i} \chi_0^2D_{jk}(\tilde{\rho} u^{i})\right\}dx=0.
\end{equation}
Here, we only show how to handle with the following terms for the first integral (\ref{int1}). While other terms for(\ref{int1}) and (\ref{int2}) could be controlled similarly as (\ref{lem6-2}).
\begin{equation}\label{lem6-1-9}
\begin{split}
(2\mu&+\lambda)\int_{\Omega}\chi_0^2\tilde{\rho}^{-2}D_{ij}q_{t} D_{ij}(\gamma\tilde{\rho}^{\gamma-2} q-\phi)dx
 +\int_{\Omega}\chi_0^2|D_{ij}(\gamma\tilde{\rho}^{\gamma-2} q-\phi)|^2dx\\[2mm]
=&\int_{\Omega}\mu\chi_0^2D_{j}\left(\tilde{\rho}^{-1}\Big(\Delta u^{i}- D_{i}(\mbox{div} u)\Big)\right)
 D_{ij}(\gamma\tilde{\rho}^{\gamma-2} q-\phi)dx
 -\int_{\Omega}\chi_0^2 D_{j}u^{i}_{t}D_{ij}(\gamma\tilde{\rho}^{\gamma-2} q-\phi)dx\\[2mm]
 &-(2\mu+\lambda)\int_{\Omega}\chi_0^2D_{j}\left(\tilde{\rho}^{-2}\Big(D_{i}\tilde{\rho}\mbox{div}u+D_{i}(u\cdot\nabla\tilde\rho)\Big)\right)
D_{ij}(\gamma\tilde{\rho}^{\gamma-2} q-\phi)dx\\[2mm]
 &+\int_{\Omega}\chi_0^2\Big\{(2\mu+\lambda)\tilde{\rho}^{-2}D_{ij}f^0+D_{j}g^{i}\Big\}\cdot D_{ij}(\gamma\tilde{\rho}^{\gamma-2} q-\phi)dx\\[2mm]
 &-(4\mu+2\lambda)\int_{\Omega}\chi_0^2\tilde{\rho}^{-2}D_{j}\tilde{\rho}D_{i}\mbox{div} uD_{ij}(\gamma\tilde{\rho}^{\gamma-2} q-\phi).
\end{split}
\end{equation}
The first term on the  left-hand side of (\ref{lem6-1-9}) becomes
\begin{equation}
\begin{split}
\int_{\Omega}&\chi_0^2\tilde{\rho}^{-2}D_{ij} q_{t} D_{ij}  (\gamma\tilde{\rho}^{\gamma-2} q-\phi)dx\\[2mm]
\geq& \frac{\gamma}{2}\frac{d}{dt}\int_{\Omega}\chi_0^2\tilde{\rho}^{\gamma-4}|D^{2} q|^2dx
-\frac{d}{dt}\int_{\Omega}\chi_0^2\tilde{\rho}^{-2} D_{ij}  q \cdot D_{ij} \phi dx\\[2mm]
&+\gamma(\gamma-2)\frac{d}{dt}\int_{\Omega}\chi_0^2\tilde{\rho}^{\gamma-5}q D_{ij}q\cdot D_{ij}\tilde{\rho} dx
+2\gamma(\gamma-2)\frac{d}{dt}\int_{\Omega}\chi_0^2\tilde{\rho}^{\gamma-5}q D_{ij}q\cdot D_{i}\tilde{\rho} D_{j}qdx\\[2mm]
&+\gamma(\gamma-2)(\gamma-3)\frac{d}{dt}\int_{\Omega}\chi_0^2\tilde{\rho}^{\gamma-6}q D_{ij}q\cdot D_i\tilde{\rho} D_{j}\tilde{\rho} dx
-\varepsilon\|D^2 q\|^2-C\Big(\|q_t\|+\|\nabla q_t\|^2\Big),
\end{split}
\end{equation}
where we have used the elliptic estimate $\|D^2\phi_t\|^2\leq C\|q_t\|^2$. The first term on the  right-hand side of (\ref{lem6-1-9}) could be bounded by
\begin{equation*}
\begin{split}
\mu&\int_{\Omega}\chi_0^2D_{j}\left(\tilde{\rho}^{-1}\Big(\Delta u^{i}- D_{i}(\mbox{div} u)\Big)\right)
 D_{ij}(\gamma\tilde{\rho}^{\gamma-2} q-\phi)dx\\[2mm]
=&\mu\int_{\Omega}\chi_0^2\tilde{\rho}^{-1}D_{j}\left(\Big(\Delta u^{i}- D_{i}(\mbox{div} u)\Big)\right)
 D_{ij}(\gamma\tilde{\rho}^{\gamma-2} q-\phi)dx\\[2mm]
 &+\mu\int_{\Omega}\chi_0^2D_{j}\tilde{\rho}^{-1}\Big(\Delta u^{i}- D_{i}(\mbox{div} u)\Big)
 D_{ij}(\gamma\tilde{\rho}^{\gamma-2} q-\phi)dx \\[2mm]
 =&-\mu\int_{\Omega}D_{k}(\chi_0^2\tilde{\rho}^{-1})(D_{jk} u^{i}-D_{ij}u^{k})  D_{ij}(\gamma\tilde{\rho}^{\gamma-2} q-\phi)dx
\\[2mm]
 &+\mu\int_{\Omega}\chi_0^2D_{j}\tilde{\rho}^{-1}\Big(\Delta u^{i}-D_{i}(\mbox{div} u)\Big)
 D_{ij}(\gamma\tilde{\rho}^{\gamma-2} q-\phi)dx \\[2mm]
\leq&  \frac{1}{4} \int_{\Omega}\chi_0^2|D^2(\gamma\tilde{\rho}^{\gamma-2} q-\phi)|^2dx+ C\int_{\Omega}|D^2 u|^2dx.
\end{split}
\end{equation*}
 At  last, the third-order derivatives (\ref{lem6-4}) can be proved similarly by estimating the following integrals:
 \begin{equation*}
\int_{\Omega}\left\{D_{ijk}(\ref{per1})_1 (2\mu+\lambda)\chi_0^2\tilde{\rho}^{-2}D_{ijk}(\gamma\tilde{\rho}^{\gamma-2} q-\phi)+ D_{jk}(\ref{per3-2})^{i} \chi_0^2D_{ijk}(\gamma\tilde{\rho}^{\gamma-2} q-\phi)\right\}dx=0,
\end{equation*}
and
\begin{equation*}
\int_{\Omega}\left\{D_{jkl}(\ref{per1})_1 \chi_0^2D_{jkl}(\gamma\tilde{\rho}^{\gamma-2} q-\phi)+ D_{jkl}(\ref{per3-2})^{i} \chi_0^2D_{jkl}(\tilde{\rho} u^{i})\right\}dx=0.
\end{equation*}
 Therefore, the proof is complete.
\end{proof}

Our next goal is to
 establish the estimates near the boundary $\partial\Omega$. For this purpose, we choose a
finite number of bounded open sets $\{O_l\}^{N}_{l=1}$ in $\mathbb{R}^3$ such that
\begin{equation*}
\partial\Omega\subset \bigcup_{j=1}^{N}O_{j},
\end{equation*}
Following the idea of \cite{MatsumuraNi1983}, local coordinates $(\xi,\zeta,r)$ will be seted up in each set $O_l$ as follows:\\
$(i)$ The boundary $O_{l}\cap\Omega$ is the image of smooth functions $z=z^i(\xi,\zeta)$ satisfying
\begin{equation*}
|z_{\xi}|=1,\quad z_{\xi}z_{\zeta}=0, \quad |z_{\zeta}|\geq \widetilde{\tau}>0,
\end{equation*}
where $\widetilde{\tau}$ is some positive constant independent of $j=1,2,\cdots, N.$\\
(ii)Any $x$ in $O_l$ is represented by
\begin{equation}{\label{transformation}}
x^{i}=x^{i}(\xi,\zeta,r)=rn^{i}(\xi,\zeta)+z^{i}(\xi,\zeta),
\end{equation}
where $n^i(\xi,\zeta)$ is the external unit normal vector at the point of the boundary coordinate $(\xi,\zeta)$. Here and in what follows we omit the suffix $l$ for simplicity. Bases on $z^i$, we introduce the unit vectors $e_1^i$ and $e_2^i$ as $e_1^i=z^i_\xi$, $e_2^i=z^i_\zeta/|z^i_\zeta|$. Thanks to Frenet's formula, there exists smooth functions $(m_1,m_2,m_3, m'_1,m'_2,m'_3)$ of $(\xi,\zeta)$ such that
\begin{equation*}
\frac{\partial}{\partial\xi}\left(\begin{matrix}e_1\\e_2\\n\end{matrix}\right)^i=\left(\begin{matrix}0&-m_3&-m_1\\
m_3&0&-m_2\\
m_1&m_2&0\end{matrix}\right)\left(\begin{matrix}e_1\\e_2\\n\end{matrix}\right)^i,
\end{equation*}
\begin{equation*}
\frac{\partial}{\partial\zeta}\left(\begin{matrix}e_1\\e_2\\n\end{matrix}\right)^i=\left(\begin{matrix}0&-m'_3&-m'_1\\
m'_3&0&-m'_2\\
m'_1&m'_2&0\end{matrix}\right)\left(\begin{matrix}e_1\\e_2\\n\end{matrix}\right)^i.
\end{equation*}

Hence the Jacobian $J$ of the transformation (\ref{transformation}) is given by
\begin{equation}{\label{jacobian}}
J=|x_\xi\times x_\zeta|=|z_\zeta|+(m_1|z_\zeta|+m_2')r+(m_1m_2'-m_2m_1')r^2.
\end{equation}
From (\ref{jacobian}), the transformation (\ref{transformation}) is regular choosing $r$ small if needed, which implies the functions $(\xi,\zeta,r)_{x_i}(x)$ make senses and we have
\begin{equation}{\label{jacobian2}}
\begin{split}
&\xi_{x_i}=\frac{1}{J}(x_\zeta\times x_r)_i=\frac{1}{J}(Ae_1^i+Be_2^i),\\
&\zeta_{x_i}=\frac{1}{J}(x_r\times x_\xi)_i=\frac{1}{J}(Ce_1^i+De_2^i),\\
&r_{x_i}=\frac{1}{J}(x_\xi\times x_\zeta)_i=n_i,
\end{split}
\end{equation}
where $A=|z_{\zeta}|+m_2'r$, $B=-m_1'r$, $C=-m_2r$, $D=1+m_1r$ and $J=AD-BC>0$. So (\ref{jacobian2}) gives us
\begin{equation*}
\partial_{x_{i}}=\frac{1}{J}\Big(Ae^i_{1}+Be^i_2\Big)\partial_{\xi}+\frac{1}{J}\Big(Ce^i_{1}+De^i_2\Big)\partial_{\zeta}+n^{i}\partial_{r}.
\end{equation*}
 Denote the tangential derivatives by $\bar{\partial}=(\partial_{\xi}, \partial_{\zeta})$, then the following estimate for the commutator hold:
\begin{equation}\label{comm}
\left\|[\partial_{x_i},\bar{\partial}]v\right\|^2\leq C \|\nabla v\|^2, \quad \mbox{for any function $v$.}
\end{equation}

 Let $\chi_{l}$ $(1\leq l\leq N)$ be any fixed cut-off function in $C^{\infty}_0(O_{l})$, we would like to derive the estimates for tangential derivatives  of order up to three.
 \begin{lem}\label{lem7}
 Assume that the conditions in Proposition \ref{prop} hold, then  for any positive $\epsilon$, it holds that
 \begin{equation}\label{lemma7}
\begin{split}
\frac{1}{2}&\frac{d}{dt} \Big\{\gamma\int_{\Omega}\chi_{l}^2\tilde\rho^{\gamma-2}|\bar{\partial}q|^2dx
+\int_{\Omega}\chi_{l}^2\tilde{\rho}|\bar{\partial} u|^2dx +\int_{\Omega}\chi_{l}^2|\bar{\partial}\nabla\phi|^2dx
\Big\}\\[2mm]
&+\frac{d}{dt}\int_{\Omega}\chi_{l}^2\bar{\partial}(\gamma\tilde{\rho}^{\gamma-2}) q \bar{\partial} q  dx
+C\Big\{\|\chi_l \bar{\partial} \nabla u\|^2+\left\|\chi_l\bar{\partial} \frac{dq}{dt}\right\|^2\Big\}\\[2mm]
\leq&\epsilon\Big(\|q\|^2+\|\nabla q\|^2+\|\nabla\phi\|^2\Big)+C\Big(\|D u\|^2+\| u_{t}\|^2+\|q_t\|^2\Big)+C\delta\mathcal{D}(t),
\end{split}
\end{equation}

\begin{equation}\label{lemma7step2}
\begin{split}
\frac{1}{2}&\frac{d}{dt} \Big\{\gamma\int_{\Omega}\chi_{l}^2\tilde\rho^{\gamma-2}|\bar{\partial}^2q|^2dx
+\int_{\Omega}\chi_{l}^2\tilde{\rho}|\bar{\partial}^2 u|^2dx +\int_{\Omega}\chi_{l}^2|\bar{\partial}^2\nabla\phi|^2dx
\Big\}\\[2mm]
&+\frac{d}{dt}\Big\{2\int_{\Omega}\chi_{l}^2\bar{\partial}(\gamma\tilde{\rho}^{\gamma-2}) \bar{\partial}q \bar{\partial}^2 q  dx
+\int_{\Omega}\chi_{l}^2\bar{\partial}^2(\gamma\tilde{\rho}^{\gamma-2}) q \bar{\partial}^2 q  dx\Big\}\\[2mm]
&+C\Big\{\|\chi_l \bar{\partial}^2 \nabla u\|^2+\left\|\chi_l\bar{\partial}^2 \frac{dq}{dt}\right\|^2\Big\}\\[2mm]
\leq&\epsilon \|D^2q\|^2+C\Big(\|q\|_{1}^2+\|D u\|_1^2+\| Du_{t}\|^2+\|q_t\|_1^2\Big)+C\delta\mathcal{D}(t),
\end{split}
\end{equation}
and
\begin{equation}\label{lemma7step3}
\begin{split}
\frac{1}{2}&\frac{d}{dt} \Big\{\gamma\int_{\Omega}\chi_{l}^2\tilde\rho^{\gamma-2}|\bar{\partial}^3q|^2dx
+\int_{\Omega}\chi_{l}^2\tilde{\rho}|\bar{\partial}^3 u|^2dx +\int_{\Omega}\chi_{l}^2|\bar{\partial}^3\nabla\phi|^2dx
\Big\}\\[2mm]
&+\frac{d}{dt}\Big\{3\int_{\Omega}\chi_{l}^2\bar{\partial}(\gamma\tilde{\rho}^{\gamma-2}) \bar{\partial}^2q \bar{\partial}^3 q  dx
+3\int_{\Omega}\chi_{l}^2\bar{\partial}^2(\gamma\tilde{\rho}^{\gamma-2})\bar{\partial} q \bar{\partial}^3 q  dx
+\int_{\Omega}\chi_{l}^2\bar{\partial}^3(\gamma\tilde{\rho}^{\gamma-2}) q \bar{\partial}^3 q  dx\Big\}\\[2mm]
&+C\Big\{\|\chi_l \bar{\partial}^3 \nabla u\|^2+\left\|\chi_l\bar{\partial}^3 \frac{dq}{dt}\right\|^2\Big\}\\[2mm]
\leq&\epsilon\|D^3q\|^2+C\Big(\|q\|_{2}^2+\|\nabla u\|_2^2+\| D^2u_{t}\|^2+\|q_t\|_2^2\Big)+C\delta\mathcal{D}(t).
\end{split}
\end{equation}
\end{lem}
\begin{proof}
Estimating the integral for
  \begin{equation}\label{tag1}
 \int_{\Omega}\Big\{\bar{\partial}(L^0-g^0)\chi_l^2\bar{\partial}(\gamma \tilde{\rho}^{\gamma-2} q)
 +\bar{\partial}(L^i-g^i)\chi_l^2\bar{\partial}(\tilde{\rho}u^{i})\Big\}dx=0.
 \end{equation}
The terms involved  $u_t$ and $q_t$ become
\begin{equation*}
 \begin{split}
 &\int_{\Omega}\chi_{l}^{2}\bar{\partial}\frac{dq}{dt}\bar{\partial}(\gamma\tilde\rho^{\gamma-2}q)dx
 \geq\int_{\Omega}\chi_{l}^{2}\bar{\partial}q_t\bar{\partial}(\gamma\tilde\rho^{\gamma-2}q)dx-C\delta\mathcal{D}(t)\\[2mm]
 &=\frac{\gamma}{2}\frac{d}{dt}\int_{\Omega}\chi_{l}^2\tilde\rho^{\gamma-2}|\bar{\partial}q|^2dx
 +\gamma\frac{d}{dt}\int_{\Omega}\chi_{l}^2\bar{\partial}(\tilde{\rho}^{\gamma-2}) q \bar{\partial} q  dx
 -\int_{\Omega}\chi_{l}^2\bar{\partial}(\tilde{\rho}^{\gamma-2}) q_t \bar{\partial} q  dx-C\delta\mathcal{D}(t)\\[2mm]
 &\geq\frac{\gamma}{2}\frac{d}{dt}\int_{\Omega}\chi_{l}^2\tilde\rho^{\gamma-2}|\bar{\partial}q|^2dx
 +\gamma\frac{d}{dt}\int_{\Omega}\chi_{l}^2\bar{\partial}(\tilde{\rho}^{\gamma-2}) q \bar{\partial} q  dx
-\epsilon \|\nabla q\|^2-C \|q_t\|^2-C\delta\mathcal{D}(t),
  \end{split}
 \end{equation*}
and
\begin{equation*}
\begin{split}
\int_{\Omega}\chi_{l}^2\bar{\partial} u^{i}_t\bar{\partial}(\tilde{\rho}u^{i})dx
&=\frac{1}{2}\frac{d}{dt}\int_{\Omega}\chi_{l}^2\tilde{\rho}|\bar{\partial} u|^2dx
+\int_{\Omega}\chi_{l}^2\bar{\partial}\tilde{\rho}\bar{\partial} u^i_t u^i dx\\[2mm]
&=\frac{1}{2}\frac{d}{dt}\int_{\Omega}\chi_{l}^2\tilde{\rho}|\bar{\partial} u|^2dx
-\int_{\Omega} u^i_t \bar{\partial}(\chi_{l}^2\bar{\partial}\tilde{\rho}u^i) dx\\[2mm]
&\geq \frac{1}{2}\frac{d}{dt}\int_{\Omega}\chi_{l}^2\tilde{\rho}|\bar{\partial} u|^2dx- C\Big\{\|\nabla u\|^2+\|u_t\|^2\Big\}.
\end{split}
\end{equation*}

Next, the following terms will be estimated by applying the inequality (\ref{comm}) for commutators, integration by part, Cauchy's inequality and Poincar$\acute{e}$'s inequality. Firstly, we have
\begin{equation*}
\begin{split}
-\int_{\Omega}\chi_{l}^2\bar{\partial}\partial_{x_i}\phi\bar{\partial}(\tilde{\rho}u^{i})dx
&=-\int_{\Omega}\chi_{l}^2\partial_{x_i}\bar{\partial}\phi\bar{\partial}(\tilde{\rho}u^{i})dx
-\int_{\Omega}\chi_{l}^2[\bar{\partial},\partial_{x_i}]\phi\bar{\partial}(\tilde{\rho}u^{i})dx\\[2mm]
&\geq\int_{\Omega}\chi_{l}^2\bar{\partial}\phi\mbox{div}(\bar{\partial}(\tilde{\rho}u^{i}))dx-\epsilon\|\nabla\phi\|^2-C\|\nabla u\|^2\\[2mm]
&\geq\int_{\Omega}\chi_{l}^2\bar{\partial}\phi\bar{\partial}\mbox{div}(\tilde{\rho}u)dx
-\epsilon\|\nabla\phi\|^2-C\|\nabla u\|^2\\[2mm]
&\geq-\int_{\Omega}\chi_{l}^2\bar{\partial}\phi\bar{\partial}q_tdx-\epsilon\|\nabla\phi\|^2-C\|\nabla u\|^2-C\delta\mathcal{D}(t)\\[2mm]
&\geq-\int_{\Omega}\chi_{l}^2\bar{\partial}\phi\bar{\partial}\Delta\phi_tdx-\epsilon\|\nabla\phi\|^2-C\|\nabla u\|^2-C\delta\mathcal{D}(t)\\[2mm]
&\geq -\int_{\Omega}\chi_{l}^2\bar{\partial}\phi\partial_{x_{j}}\bar{\partial}\partial_{x_j}\phi_tdx
-\int_{\Omega}\chi_{l}^2\bar{\partial}\phi[\bar{\partial},\partial_{x_{j}}]\partial_{x_j}\phi_tdx
-\epsilon\|\nabla\phi\|^2-C\|\nabla u\|^2-C\delta\mathcal{D}(t)\\[2mm]
&\geq \int_{\Omega}\chi_{l}^2\partial_{x_{j}}\bar{\partial}\phi\bar{\partial}\partial_{x_j}\phi_tdx
-\epsilon\|\nabla\phi\|^2-C\|\nabla u\|^2-C\|q_t\|^2-C\delta\mathcal{D}(t)\\[2mm]
&\geq\frac{1}{2}\frac{d}{dt}\int_{\Omega}\chi_{l}^2|\bar{\partial}\nabla\phi|^2dx-\epsilon\|\nabla\phi\|^2-C\|\nabla u\|^2-C\|q_t\|^2-
C\delta\mathcal{D}(t),
\end{split}
\end{equation*}
where we have used the fact $\left\|[\bar{\partial},\partial_{x_{j}}]\partial_{x_j}\phi_t\right\|^2\leq C\|\nabla^{2}\phi_t\|^2\leq C\|q_t\|^2$.
Similarly, one can get
\begin{equation*}
\begin{split}
\int_{\Omega}\chi_{l}^2\bar{\partial}\mbox{div}(\tilde{\rho}u)\bar{\partial}(\gamma\tilde{\rho}^{\gamma-2}q)dx
+\int_{\Omega}\chi_{l}^2\bar{\partial}\nabla(\gamma \tilde{\rho}^{\gamma-2} q)
\cdot\bar{\partial}(\tilde{\rho}u)dx
\geq -\epsilon(\|q\|^2+\|\nabla q\|^2)-C\|\nabla u\|^2.
\end{split}
\end{equation*}
While the principal terms could be dealt with as follows:
\begin{equation*}\label{L1}
\begin{split}
-&\int_{\Omega}\chi_{l}^2\bar{\partial} \Delta (\tilde{\rho}^{-1}u^{i})\bar{\partial}(\tilde{\rho}u^{i})dx\\[2mm]
=&-\int_{\Omega}\chi_{l}^2\partial_{x_j}\bar{\partial}\partial_{x_j}(\tilde{\rho}^{-1}u^{i})\bar{\partial}(\tilde{\rho}u^{i})dx
-\int_{\Omega}\chi_{l}^2[\bar{\partial},\partial_{x_j}]\partial_{x_j}(\tilde{\rho}^{-1}u^{i})\bar{\partial}(\tilde{\rho}u^{i})dx\\[2mm]
\geq&\int_{\Omega}\chi_{l}^2\bar{\partial}\partial_{x_j}(\tilde{\rho}^{-1}u^{i})\partial_{x_j}\bar{\partial}(\tilde{\rho}u^{i})dx
+\int_{\Omega}\chi_{l}^2\partial_{x_j}(\tilde{\rho}^{-1}u^{i})[\bar{\partial},\partial_{x_j}]\bar{\partial}(\tilde{\rho}u^{i})dx
-\varepsilon\|\chi_{l}\bar{\partial}\nabla u \|^2-C\|\nabla u\|^2\\[2mm]
\geq&\int_{\Omega}\chi_{l}^2\bar{\partial}\partial_{x_j}(\tilde{\rho}^{-1}u^{i})\bar{\partial}\partial_{x_j}(\tilde{\rho}u^{i})dx
-\int_{\Omega}\chi_{l}^2\bar{\partial}\partial_{x_j}(\tilde{\rho}^{-1}u^{i})[\bar{\partial},\partial_{x_j}](\tilde{\rho}u^{i})dx\\[2mm]
&+\int_{\Omega}\chi_{l}^2\partial_{x_j}(\tilde{\rho}^{-1}u^{i})[\bar{\partial},\partial_{x_j}]\bar{\partial}(\tilde{\rho}u^{i})dx
-\varepsilon\|\chi_{l}\bar{\partial}\nabla u \|^2-C\|\nabla u\|^2\\[2mm]
\geq& \frac{1}{2}\int_{\Omega}\chi_{l}^2|\bar{\partial}\nabla u|^2 dx-C\|\nabla u\|^2,
\end{split}
\end{equation*}
and
\begin{equation*}
\begin{split}
-\int_{\Omega}\chi_{l}^2\bar{\partial}\left(\tilde{\rho}^{-1} \partial_{x_i}\mbox{div}u\right)\bar{\partial}(\tilde{\rho}u^{i})dx
&\geq \frac{1}{2}\int_{\Omega}\chi_{l}^2|\bar{\partial}\mbox{div} u|^2 dx-C\|\nabla u\|^2\\[2mm]
&\geq C\int_{\Omega}\chi_{l}^2|\bar{\partial}\frac{dq}{dt}|^2 dx-C\|\nabla u\|^2.
\end{split}
\end{equation*}
Finally, the rest of nonlinar terms for (\ref{tag1}) could be dominated by $C\delta\mathcal{D}(t)$. Therefore, combining the above inequalities with (\ref{tag1}), it induces the desired result (\ref{lemma7}).

The proofs for  (\ref{lemma7step2}) and (\ref{lemma7step3}) are achieved by estimating the following integrals
  \begin{equation*}
 \int_{\Omega}\Big\{\bar{\partial}^{2}(L^0-g^0)\chi_l^2\bar{\partial}^{2}(\gamma \tilde{\rho}^{\gamma-2} q)
 +\bar{\partial}^2(L-g)\cdot\chi_l^2\bar{\partial}^2(\tilde{\rho}u)\Big\}dx=0,
 \end{equation*}
 \begin{equation*}
 \int_{\Omega}\Big\{\bar{\partial}^{3}(L^0-g^0)\chi_l^2\bar{\partial}^{3}(\gamma \tilde{\rho}^{\gamma-2} q)
 +\bar{\partial}^2(L-g)\cdot\chi_l^2\bar{\partial}^{3}(\tilde{\rho}u)\Big\}dx=0,
 \end{equation*}
respectively, and utilizing the similar argument as (\ref{lemma7}). Thus, this lemma has been completed.
\end{proof}
\vspace{2mm}

Next,   we estimate the normal derivatives and the mixed  (tangential-normal) derivatives of the solutions. for this purpose, in each $O_j$, rewriting the equations (\ref{per3-1}), (\ref{per3-2})
  by local coordinates $(\xi,\zeta,r)$ as
\begin{equation*}
\begin{split}
\bar{L}^{0}\equiv&\frac{d q}{dt}+\Big\{\frac{1}{J}(Ae^i_{1}+Be^i_2)(\tilde{\rho}u^i)_{\xi}+\frac{1}{J}(Ce^i_{1}+De^i_2)(\tilde{\rho}u^i)_{\zeta}
+n^{i}(\tilde{\rho}u^i)_{r}\Big\}=g^{0},\\[3mm]
\bar{L}^{i}\equiv&u_{t}^{i}-\mu\frac{1}{\tilde{\rho}}\Big\{\frac{1}{J^2}(A^2+B^2)u^{i}_{\xi\xi}+\frac{2}{J^2}(AC+BD)u^{i}_{\xi\zeta}
+\frac{1}{J^2}(C^2+D^2)u^{i}_{\zeta\zeta}+u^{i}_{rr}\Big\}\\[2mm]
&+\mbox{first order and zero order terms of $u$}\\[2mm]
&+(\mu+\lambda)\frac{1}{\tilde{\rho}}\Big\{\frac{1}{J}(Ae^{i}_1+Be^{i}_2)\Big(\frac{1}{\tilde{\rho}}\frac{dq}{dt}\Big)_{\xi}
+\frac{1}{J}(Ce^{i}_1+De^{i}_2)\Big(\frac{1}{\tilde{\rho}}\frac{dq}{dt}\Big)_{\zeta}
+n^{i}\Big(\frac{1}{\tilde{\rho}}\frac{dq}{dt}\Big)_{r}\Big\}\\[2mm]
&+\left\{\frac{1}{J}(Ae^{i}_1+Be^{i}_2)\Big(\gamma \tilde{\rho}^{\gamma-2} q-\phi\Big)_{\xi}
+\frac{1}{J}(Ce^{i}_1+De^{i}_2)\Big(\gamma \tilde{\rho}^{\gamma-2} q-\phi\Big)_{\zeta}
+n^{i}\Big(\gamma\tilde{\rho}^{\gamma-2} q-\phi\Big)_{r}\right\}\\[2mm]
=&(\mu+\lambda)\frac{1}{\tilde{\rho}}\left(\frac{1}{\tilde{\rho}}g^0\right)_{x_i}+g^{i}, \quad i=1,2,3,
\end{split}
\end{equation*}
where we have used $\mbox{div}u=\tilde{\rho}^{-1}\Big(g^{0}-u\cdot\nabla\tilde{\rho}-\frac{dq}{dt}\Big)$.

One can rewrite $\partial_{r}(\bar{L}^0-g^0)=0$ and $n^{i}(\bar{L}^i-g^i)=0$ as:
\begin{equation}\label{norm1}
\begin{split}
&\left(\frac{d q}{dt}\right)_{r}+\Big\{\frac{1}{J}(Ae^i_{1}+Be^i_2)\tilde{\rho}u^i_{r\xi}+\frac{1}{J}(Ce^i_{1}+De^i_2)\tilde{\rho}u^i_{r\zeta}
+n^{i}\tilde{\rho}u^i_{rr}\Big\}\\[2mm]
&+\mbox{first order and zero order terms  of $u$}=g_{r}^0,
\end{split}
\end{equation}
and
\begin{equation}\label{norm2}
\begin{split}
&n^{i}u_{t}^{i}-\mu\frac{1}{\tilde{\rho}}\Big\{\frac{1}{J^2}(A^2+B^2)n^{i}u^{i}_{\xi\xi}+\frac{2}{J^2}(AC+BD)n^{i}u^{i}_{\xi\zeta}
+\frac{1}{J^2}(C^2+D^2)n^{i}u^{i}_{\zeta\zeta}+n^{i}u^{i}_{rr}\Big\}\\[2mm]
&+\mbox{first order and zero order  terms of $u$}\\[2mm]
&+(\mu+\lambda)\frac{1}{\tilde{\rho}}\Big(\frac{1}{\tilde{\rho}}\frac{dq}{dt}\Big)_{r}+(\gamma \tilde{\rho}^{\gamma-2}q-\phi)_{r}=(\mu+\lambda)\frac{1}{\tilde{\rho}}\left(\frac{1}{\tilde{\rho}}g^0\right)_{r}+n^{i}g^{i},
\end{split}
\end{equation}
where we have used the following equality
\begin{equation*}
\left(\frac{1}{\tilde{\rho}}g^0\right)_{x_i}=\frac{1}{J}(Ae^i_{1}+Be^i_2)\left(\frac{1}{\tilde{\rho}}g^0\right)_{\xi}
+\frac{1}{J}(Ce^i_{1}+De^i_2)\left(\frac{1}{\tilde{\rho}}g^0\right)_{\zeta}
+n^{i}\left(\frac{1}{\tilde{\rho}}g^0\right)_{r}.
\end{equation*}
Eliminating the  term $n^{i}u^{i}_{rr}$ from (\ref{norm1}) and (\ref{norm2}), one has
\begin{equation}\label{norm3}
\begin{split}
(2\mu&+\lambda)\frac{1}{\tilde\rho^2}\left(\frac{d q}{dt}\right)_{r}
+(\gamma \tilde{\rho}^{\gamma-2}q-\phi)_{r}\\
=&-n^{i}u_{t}^{i}+\frac{\mu}{\tilde{\rho}}\Big\{\frac{1}{J^2}(A^2+B^2)n^{i}u^{i}_{\xi\xi}+\frac{2}{J^2}(AC+BD)n^{i}u^{i}_{\xi\zeta}
+\frac{1}{J^2}(C^2+D^2)n^{i}u^{i}_{\zeta\zeta}\\[2mm]
&-\frac{1}{J}(Ae^i_{1}+Be^i_2)u^{i}_{r\xi}
-\frac{1}{J}(Ce^i_{1}+De^i_2)u^{i}_{r\zeta}\Big\}\\[2mm]
&+\mbox{first order and zero terms of $u$}
+\mu \frac{1}{\tilde\rho^2} g_r^0+(\mu+\lambda)\frac{1}{\tilde{\rho}}\left(\frac{1}{\tilde{\rho}}g^0\right)_{r}+n^{i}g^{i}.
\end{split}
\end{equation}

\begin{lem}\label{lem8}
Under the assumptions in Proposition \ref{prop}, then  for any positive $\epsilon$, it holds that

(i)\ the estimate of normal derivative:
 \begin{equation}\label{lemma8}
\begin{split}
\frac{\gamma}{2}&\frac{d}{dt}\int_{\Omega}\chi_{l}^2\tilde\rho^{\gamma-4}|q_r|^2dx
-\frac{d}{dt}\Big\{
\int_{\Omega}\chi_{l}^2\tilde\rho^{-2}q_{r}\phi_{r}dx
-\int_{\Omega}\chi_{l}^2\tilde\rho^{-2}(\gamma\tilde\rho^{\gamma-2})_{r}q_{r}q dx
\Big\}\\[2mm]
&+C\left\{\|\chi_l (\gamma \tilde\rho^{\gamma-2}q-\phi)_{r}\|^2
+\left\|\chi_l\left(\frac{dq}{dt}\right)_{r}\right\|^2\right\}\\[2mm]
\leq&\epsilon\|Dq\|^2+C\Big(\|q_t\|^2+\|D u\|^2+\|u_t\|^2+\|\chi_l\bar{\partial}D u\|^2\Big)+C\delta\mathcal{D}(t).
\end{split}
\end{equation}
(ii)\ For $k+m=1$, it holds that
\begin{equation}\label{lemma8step2}
\begin{split}
\frac{d}{dt}&\frac{\gamma}{2}\int_{\Omega}\chi_{l}^2\tilde\rho^{\gamma-4}|\bar{\partial}^{k}\partial_{r}^{m}q_r|^2dx
-\frac{d}{dt}\Big\{\int_{\Omega}\chi_{l}^2\tilde\rho^{-2}\bar{\partial}^{k}\partial_{r}^{m}q_{r}\bar{\partial}^{k}\partial_{r}^{m}\phi_{r}dx\\[2mm]
&-\int_{\Omega}\chi_{l}^2 \tilde\rho^{-2}\bar{\partial}^{k}\partial_{r}^{m}q_{r}
\Big(\bar{\partial}^{k}\partial_{r}^{m}(\gamma\tilde\rho^{\gamma-2})q_{r}+ (\gamma\tilde\rho^{\gamma-2})_r \bar{\partial}^{k}\partial_{r}^{m}q
+\bar{\partial}^{k}\partial_{r}^{m}(\gamma\tilde\rho^{\gamma-2})_rqdx\Big) \Big\}\\[2mm]
&+C\left\{\|\chi_l\bar{\partial}^{k}\partial_{r}^{m} (\gamma \tilde\rho^{\gamma-2}q-\phi)_{r}\|^2
+\left\|\chi_l\bar{\partial}^{k}\partial_{r}^{m}\left(\frac{dq}{dt}\right)_{r}\right\|^2\right\}\\[2mm]
\leq&C\left(\left\|\left(\frac{dq}{dt}\right)_{r}\right\|^2+\|q\|_{1}^2+\|q_t\|_1^2
+\|Du_t\|^2+\|Du\|_{1}^2
+\|\chi_{l}\bar{\partial}^{k+1}\partial_{r}^{m}Du\|^2\right)\\[2mm]
&+\epsilon\|D^2q\|^2+C\delta\mathcal{D}(t).
\end{split}
\end{equation}
(iii)\ For $k+m=2$, it holds that
\begin{equation}\label{lemma8step3}
\begin{split}
\frac{\gamma}{2}&\frac{d}{dt}\int_{\Omega}\chi_{l}^2\tilde\rho^{\gamma-4}|\bar{\partial}^{k}\partial_{r}^{m}q_r|^2dx
-\frac{d}{dt}\Big\{\int_{\Omega}\chi_{l}^2\tilde\rho^{-2}\bar{\partial}^{k}\partial_{r}^{m}q_{r}\bar{\partial}^{k}\partial_{r}^{m}\phi_{r}dx
-\int_{\Omega}G_{k,m}dx \Big\}\\[2mm]
&+C\left\{\|\chi_l\bar{\partial}^{k}\partial_{r}^{m} (\gamma \tilde\rho^{\gamma-2}q-\phi)_{r}\|^2
+\left\|\chi_l\bar{\partial}^{k}\partial_{r}^{m}\left(\frac{dq}{dt}\right)_{r}\right\|^2\right\}\\[2mm]
\leq&C\left(\left\|\left(\frac{dq}{dt}\right)_{r}\right\|_1^2+\|q\|_2^2+\|q_t\|_2^2+\|D^2u_t\|^2+\|Du\|_{2}^2
+\|\chi_{l}\bar{\partial}^{k}\partial_{r}^{m}\bar{\partial}Du\|^2\right)\\[2mm]
&+\epsilon\|D^3q\|^2+C\delta\mathcal{D}(t),
\end{split}
\end{equation}
where
\begin{equation*}
\begin{split}
&G_{2,0}=\chi_{l}^2 \tilde\rho^{-2}\bar{\partial}^{2}q_{r}\Big\{
(\gamma\tilde\rho^{\gamma-2})_r\bar{\partial}^{2}q
+2\bar{\partial}(\gamma\tilde{\rho}^{\gamma-2})\bar{\partial}q_{r}
+\bar{\partial}^{2}(\gamma\tilde\rho^{\gamma-2})q_{r}+\bar{\partial}(\gamma\rho^{\gamma-2})_{r}\bar{\partial}q
+\bar{\partial}^{2}(\gamma\tilde\rho^{\gamma-2})_rqdx\Big\},\\[2mm]
&G_{1,1}=\chi_{l}^2 \tilde\rho^{-2}\bar{\partial}q_{rr}\Big\{
\bar{\partial}(\gamma\tilde{\rho}^{\gamma-2})q_{rr}
2(\gamma\tilde\rho^{\gamma-2})_r\bar{\partial}q_r
+\bar{\partial}(\gamma\tilde\rho^{\gamma-2})_{r}q_{r}+(\gamma\rho^{\gamma-2})_{rr}\bar{\partial}q
+\bar{\partial}(\gamma\tilde\rho^{\gamma-2})_{rr}qdx\Big\},\\[2mm]
&G_{0,2}=\chi_{l}^2 \tilde\rho^{-2}q_{rrr}
\Big\{
3(\gamma\tilde\rho^{\gamma-2})_{r}q_{rr}
+3(\gamma\tilde{\rho}^{\gamma-2})_{rr}q_{r}
+(\gamma\tilde\rho^{\gamma-2})_{rrr}q\Big\}.
\end{split}
\end{equation*}

For simply, we denote  the estimate of normal-normal derivative (i.e. (\ref{lemma8step2}) with $k=0$, $m=1$) by
\begin{equation}\label{h3}
\begin{split}
\frac{d}{dt}&(H_3+F_3)+C\left\{\|\chi_l(\gamma \tilde\rho^{\gamma-2}q-\phi)_{rr}\|^2
+\left\|\chi_l\left(\frac{dq}{dt}\right)_{rr}\right\|^2\right\}\\[2mm]
\leq&C\left(\left\|\left(\frac{dq}{dt}\right)_{r}\right\|^2+\|q\|_{1}^2+\|q_t\|_1^2
+\|Du_t\|^2+\|Du\|_{1}^2
+\|\chi_{l}\bar{\partial}\partial_{r}Du\|^2\right)\\[2mm]
&+\epsilon\|D^2q\|^2+C\delta\mathcal{D}(t).
\end{split}
\end{equation}
Denote the estimate of  tangential-normal-normal derivative (i.e. (\ref{lemma8step3}) with $k=1$, $m=1$) by
\begin{equation}
\begin{split}
\frac{d}{dt}&(H_5+F_5)+C\Big\{\|\chi_l\bar{\partial}(\gamma \tilde\rho^{\gamma-2}q-\phi)_{rr}\|^2
+\left\|\chi_l\bar{\partial}\left(\frac{dq}{dt}\right)_{rr}\right\|^2\Big\}\\[2mm]
\leq&C\left(\left\|\left(\frac{dq}{dt}\right)_{r}\right\|_1^2+\|q\|_2^2+\|q_t\|_2^2+\|D^2u_t\|^2+\|Du\|_{2}^2
+\|\chi_{l}\bar{\partial}^2\partial_{r}Du\|^2\right)\\[2mm]
&+\epsilon\|D^3q\|^2+C\delta\mathcal{D}(t),
\end{split}
\end{equation}
and the estimate of  normal-normal-normal derivative (i.e. (\ref{lemma8step3}) with $k=0$, $m=2$) by
\begin{equation}\label{step6}
\begin{split}
\frac{d}{dt}&(H_6+F_6)+C\Big\{\|\chi_l\partial_{r}^{3} (\gamma \tilde\rho^{\gamma-2}q-\phi)\|^2
+\left\|\chi_l\partial_{r}^{3}\left(\frac{dq}{dt}\right)\right\|^2\Big\}\\[2mm]
\leq& C\left(\left\|\left(\frac{dq}{dt}\right)_{r}\right\|_1^2+\|q\|_2^2+\|q_t\|_2^2+\|D^2u_t\|^2+\|Du\|_{2}^2
+\|\chi_{l}\partial_{r}^{2}\bar{\partial}Du\|^2\right)\\[2mm]
&+\epsilon\|D^3q\|^2+C\delta\mathcal{D}(t).
\end{split}
\end{equation}
\end{lem}

 \begin{proof}
 Taking the inner product of (\ref{norm3}) with $\chi_{l}^2(\gamma\tilde{\rho}^{\gamma-2} q-\phi)_{r}$, then the left-hand side is
 \begin{equation}\label{norm4}
\begin{split}
&(2\mu+\lambda)\int_{\Omega}\chi_{l}^2\frac{1}{\tilde\rho^2}\left(\frac{d q}{dt}\right)_{r} (\gamma\tilde{\rho}^{\gamma-2} q-\phi)_{r}dx
+\|\chi_{l}(\gamma \tilde{\rho}^{\gamma-2}q-\phi)_{r}\|^2\\
&\equiv LHS_1+LHS_2+\|\chi_{l}(\gamma \tilde{\rho}^{\gamma-2}q-\phi)_{r}\|^2.
\end{split}
\end{equation}
A simple calculation gives
\begin{equation*}
\begin{split}
LHS_1\equiv&\int_{\Omega}\chi_{l}^2\frac{1}{\tilde\rho^2}\left(\frac{d q}{dt}\right)_{r} (\gamma\tilde{\rho}^{\gamma-2} q)_{r}dx\\[2mm]
=&\gamma\int_{\Omega}\chi_{l}^2\tilde\rho^{\gamma-4}q_{tr}q_r dx
 +\int_{\Omega}\chi_{l}^2\tilde\rho^{-2}(\gamma\tilde\rho^{\gamma-2})_{r}q_{tr}q dx
+\int_{\Omega}\chi_{l}^2\tilde\rho^{-2}(u\cdot\nabla q)_{r} (\gamma\tilde{\rho}^{\gamma-2} q)_{r}dx\\[2mm]
 \geq&\frac{\gamma}{2}\frac{d}{dt}\int_{\Omega}\chi_{l}^2\tilde\rho^{\gamma-4}|q_r|^2dx
+\frac{d}{dt}\int_{\Omega}\chi_{l}^2\tilde\rho^{-2}(\gamma\tilde\rho^{\gamma-2})_{r}q_{r}q dx-\epsilon\|\nabla q\|^2
-C\|q_t\|^2-C\delta\mathcal{D}(t),
\end{split}
\end{equation*}
and
\begin{equation*}
\begin{split}
LHS_2\equiv&-\int_{\Omega}\chi_{l}^2\frac{1}{\tilde\rho^2}\left(\frac{d q}{dt}\right)_{r} \phi_{r}dx\\[2mm]
=&-\frac{d}{dt}\int_{\Omega}\chi_{l}^2{\tilde\rho^{-2}}q_{r}\phi_{r}dx+\int_{\Omega}\chi_{l}^2{\tilde\rho^{-2}}q_{r}\phi_{tr}dx
-\int_{\Omega}\chi_{l}^2\tilde\rho^{-2}(u\cdot\nabla q)_{r}\phi_{r}dx\\[2mm]
\geq&-\frac{d}{dt}\int_{\Omega}\chi_{l}^2\tilde\rho^{-2}q_{r}\phi_{r}dx-\epsilon\|\nabla q\|^2-C\|\nabla u\|^2-C\delta\mathcal{D}(t),
\end{split}
\end{equation*}
where we have used (\ref{imp1}) in the last inequality.

Putting the above two inequalities with (\ref{norm4}), and using Cauchy's inequality, one obtains that
\begin{equation}\label{norm5}
\begin{split}
\frac{d}{dt}&\Big\{\frac{\gamma}{2}\int_{\Omega}\chi_{l}^2\tilde\rho^{\gamma-4}|q_r|^2dx
+\gamma(\gamma-2)\int_{\Omega}\chi_{l}^2\tilde\rho^{\gamma-5}\tilde{\rho}_{r}q_{r}qdx
-\int_{\Omega}\chi_{l}^2\tilde\rho^{-2}q_{r}\phi_{r}dx\Big\}\\[2mm]
&+\frac{1}{2}\|\chi_l (\gamma \tilde\rho^{\gamma-2}q-\phi)_{r}\|^2\\[2mm]
\leq&\epsilon\|\nabla q\|^2+C\Big(\|q_t\|^2+\|\nabla u\|^2+\|u_t\|^2+\|\chi_l\bar{\partial} D u\|^2\Big)+C\delta\mathcal{D}(t).
\end{split}
\end{equation}
Meanwhile, taking the inner product of (\ref{norm3}) with $\chi_{l}^2\left(\frac{dq}{dt}\right)_{r}$, a similar argument as (\ref{norm5}) gives the desired estimate (\ref{lemma8}). Furthermore, the estimates (\ref{lemma8step2}) and  (\ref{lemma8step3}) can be obtained
in a similar way as before, we thus omit the proof.
\end{proof}

At last,  taking  $\chi_{l}\bar{\partial}_{k}( k=1,2)$ to equations (\ref{per3-1}), (\ref{per3-2}) in a similar manner to Lemma \ref{lemmstoke},    one obtains
\begin{lem}\label{lem10}
For  $k=1,2$ and   $k+m=1,2$, it holds
\begin{equation*}
\begin{split}
&\|\chi_lD^{m+2}\bar{\partial}^{k}u\|^2+\|\chi_lD^{m+1}\bar{\partial}^{k}(\gamma \tilde{\rho}^{\gamma-2}q-\phi)\|^2\\[1mm]
&\leq C\left(\left\|\chi_l\bar{\partial}^{k}\frac{dq}{dt}\right\|^2_{{m+1}}+
\|g^0\|^2_{{k+m+1}}+\|u_t\|^2_{k+m}+\|g\|^2_{{k+m}}+\|D q\|^2_{{k+m-1}}+\|Du\|^2_{{k+m}}\right).
\end{split}
\end{equation*}
\end{lem}
\vspace{1mm}

Now we are ready to prove Proposition  \ref{prop} by the following steps.
\begin{proof} [\textbf{Proof of Proposition \ref{prop}}]

{\bf Step 1.}  Adding the results of Lemma \ref{lemma1}, Lemma \ref{lemma3},  (\ref{lem6-2}), (\ref{lemma7}),
and (\ref{lemma8}) with some small suitable constants, then one obtains that
\begin{equation}\label{step1}
\begin{split}
\frac{1}{2}&\frac{d}{dt}\Big\{\int_{\Omega}\Big(\rho|u|^2+\gamma\tilde{\rho}^{\gamma-2}q^2+|\nabla\phi|^2\Big)dx
+\int_{\Omega}\Big(\mu|\nabla u|^2+(\mu+\lambda)|\mbox{div} u|^2\Big)dx
+\gamma\int_{\Omega}\chi_0^2\tilde{\rho}^{\gamma-4}|\nabla q|^2dx\\[2mm]
&+\frac{\gamma}{2}\int_{\Omega}\chi_{l}^2\tilde\rho^{\gamma-2}|\bar{\partial}q|^2dx
+\int_{\Omega}\chi_{l}^2\tilde{\rho}|\bar{\partial} u|^2dx +\int_{\Omega}\chi_{l}^2|\bar{\partial}\nabla\phi|^2dx
+\frac{\gamma}{2}\int_{\Omega}\chi_{l}^2\tilde\rho^{\gamma-4}|q_r|^2dx
\Big\}\\[2mm]
&-\frac{d}{dt}\Big\{\int_{\Omega}\tilde{\rho}\nabla\phi\cdot u dx+\int_{\Omega}q u\cdot\nabla\tilde{\Phi}dx
+\int_{\Omega}\gamma\tilde{\rho}^{\gamma-1}q\mbox{div}udx
 -\int_{\Omega}\chi_0^2\tilde{\rho}^{-2}q\nabla q\cdot\nabla(\gamma\tilde{\rho}^{\gamma-2}) dx\\[2mm]
&+\int_{\Omega}\chi_0^2\tilde{\rho}^{-2} \nabla q \cdot \nabla\phi dx
 -\int_{\Omega}\chi_{l}^2\bar{\partial}(\gamma\tilde{\rho}^{\gamma-2}) q \bar{\partial} q  dx
+\int_{\Omega}\chi_{l}^2\tilde\rho^{-2}(\gamma\tilde\rho^{\gamma-2})_{r}q_{r}q dx\\[2mm]
&+\int_{\Omega}\chi_{l}^2\tilde\rho^{-2}q_{r}\phi_{r}dx\Big\}+C\left\{
\|D u\|^2+\|q_t\|^2+\|u_t\|^2+\left\|\frac{dq}{dt}\right\|^2_{1}\right\}\\[2mm]
\leq& \epsilon (\|q\|^2+\|\nabla q\|^2)+C\delta\mathcal{D}(t).
\end{split}
\end{equation}
Further, utilizing Lemma \ref{lemmstoke} with $k=0$, Lemma \ref{lemma5} and Poincar$\acute{e}$'s inequality, we have
\begin{equation*}
\begin{split}
&\|q\|^2+\|\nabla q\|^2+\|\nabla\phi\|^2\leq C\left\{\|D^{2}u\|^2+\|D(\gamma\tilde\rho^{\gamma-2}q-\phi)\|^2+\delta\mathcal{D}(t)\right\}\\[2mm]
&\leq C\left\{\left\|\frac{dq}{dt}\right\|^2_{1}+\|u\|^2_{1}+\|u_t\|^2+\delta\mathcal{D}(t)\right\}\\[2mm]
&\leq C\left\{\left\|\frac{dq}{dt}\right\|^2_{1}+\|\nabla u\|^2+\|u_t\|^2+\delta\mathcal{D}(t)\right\}.
\end{split}
\end{equation*}
Denoting the time derivative  of (\ref{step1}) by  $\frac{d}{dt}H_{1}(t)+\frac{d}{dt}F_1(t)$, then for $\epsilon$ small enough, (\ref{step1})  gives
\begin{equation}\label{step1-1}
\begin{split}
\frac{d}{dt}\Big(H_{1}(t)+F_1(t)\Big)
+C\left\{\|q\|_{1}^2+\|q_t\|_{1}^2+\|\nabla\phi\|^2+
\|D u\|_{1}^2+\|u_t\|^2+\left\|\frac{dq}{dt}\right\|^2_{1}\right\}\leq C\delta\mathcal{D}(t).
\end{split}
\end{equation}

{\bf Step 2.} In view of Lemma \ref{lemma2}, Lemma \ref {lemma6}, Lemma \ref{lem7} and  Lemma \ref{lem8} with $k=1$, $m=0$, one has

\begin{equation}\label{step2}
\begin{split}
\frac{1}{2}&\frac{d}{dt}\Big\{\int_{\Omega}\Big(\rho|u_t|^2+\gamma\tilde{\rho}^{\gamma-2}q_t^2+|D\phi_t|^2\Big)dx+
\gamma\int_{\Omega}\chi_0^2\tilde{\rho}^{\gamma-4}|D^{2} q|^2dx\\[2mm]
&+\gamma\int_{\Omega}\chi_{l}^2\tilde\rho^{\gamma-2}|\bar{\partial}^2q|^2dx
+\int_{\Omega}\chi_{l}^2\tilde{\rho}|\bar{\partial}^2 u|^2dx +\int_{\Omega}\chi_{l}^2|\bar{\partial}^2D\phi|^2dx
+\gamma\int_{\Omega}\chi_{l}^2\tilde\rho^{\gamma-4}|\bar{\partial}q_r|^2dx\Big\}\\[2mm]
&+\frac{d}{dt}\Big\{
-\int_{\Omega}\chi_0^2\tilde{\rho}^{-2} D_{ij}q \cdot D_{ij} \phi dx
+2\int_{\Omega}\chi_0^2\tilde{\rho}^{-2}D_{i}(\gamma\tilde{\rho}^{\gamma-2})\cdot D_{ij}q D_{j}q dx
+\int_{\Omega}\chi_0^2\tilde{\rho}^{-2}D_{ij}(\gamma\tilde{\rho}^{\gamma-2})D_{ij}q q dx\\[2mm]
&+2\int_{\Omega}\chi_{l}^2\bar{\partial}(\gamma\tilde{\rho}^{\gamma-2}) \bar{\partial}q \bar{\partial}^2 q  dx
+\int_{\Omega}\chi_{l}^2\bar{\partial}^2(\gamma\tilde{\rho}^{\gamma-2}) q \bar{\partial}^2 q  dx
-\int_{\Omega}\chi_{l}^2\tilde\rho^{-2}\bar{\partial}q_{r}\bar{\partial}\phi_{r}dx\\[2mm]
&+\int_{\Omega}\chi_{l}^2 \tilde\rho^{-2}\bar{\partial}q_{r}
\Big(\bar{\partial}(\gamma\tilde\rho^{\gamma-2})q_{r}+ (\gamma\tilde\rho^{\gamma-2})_r \bar{\partial}q
+\bar{\partial}(\gamma\tilde\rho^{\gamma-2})_rqdx\Big) \Big\}\\[3mm]
&+C\Big\{\|D u_t\|^2+\left\|\left(\frac{dq}{dt}\right)_t\right\|^2
+\|\chi_0D^{2}(\gamma\tilde{\rho}^{\gamma-2}q-\phi)\|^2+\|\chi_0D^3u\|^2
+\left\|\chi_0 D^2\frac{dq}{dt}\right\|^2\\[2mm]
&+\|\chi_l \bar{\partial}^2 D u\|^2+\|\chi_l\bar{\partial}^2 \frac{dq}{dt}\|^2
+\|\chi_l\bar{\partial}(\gamma \tilde\rho^{\gamma-2}q-\phi)_{r}\|^2
+\left\|\chi_l\bar{\partial}\left(\frac{dq}{dt}\right)_{r}\right\|^2\Big\}\\[2mm]
\leq&\epsilon \|D^2 q\|^2+C\Big\{\|q\|_{1}^2+\|q_t\|_1^2+\|Du\|_{1}^2+\| Du_{t}\|^2
+\left\|\left(\frac{dq}{dt}\right)_{r}\right\|^2
+\|\chi_{l}\bar{\partial}^2Du\|^2
\Big\}
+C\delta\mathcal{D}(t).
\end{split}
\end{equation}
Denoting the time derivative  of (\ref{step2}) by  $\frac{d}{dt}H_{2}(t)+\frac{d}{dt}F_2(t)$, then (\ref{step2}) and (\ref{step1-1})  infer that
\begin{equation}\label{step2-1}
\begin{split}
\frac{d}{dt}&\Big\{H_{1}(t)+F_1(t)+\eta_2H_2(t)+\eta_2F_{2}(t)\Big\}
+C\Big\{\|q\|_{1}^2+\|q_t\|_{1}^2+\|\nabla\phi\|^2+
\|D u\|_{1}^2+\|u_t\|_{1}^2\\[2mm]
&+\left\|\frac{dq}{dt}\right\|^2_{1}+\left\|\left(\frac{dq}{dt}\right)_t\right\|^2
+\|\chi_0D^{2}(\gamma\tilde{\rho}^{\gamma-2}q-\phi)\|^2+\|\chi_0D^3u\|^2
+\left\|\chi_0 D^2\frac{dq}{dt}\right\|^2\\[2mm]
&+\|\chi_l \bar{\partial}^2 D u\|^2+\|\chi_l\bar{\partial}^2 \frac{dq}{dt}\|^2
+\|\chi_l\bar{\partial}(\gamma \tilde\rho^{\gamma-2}q-\phi)_{r}\|^2
+\left\|\chi_l\bar{\partial}\left(\frac{dq}{dt}\right)_{r}\right\|^2\Big\}\\[2mm]
\leq& \epsilon \|D^2 q\|^2+C\delta\mathcal{D}(t).
\end{split}
\end{equation}

{\bf Step 3.} Lemma \ref{lem8} with $k=0$, $m=1$, i.e. (\ref{h3}) tells that

\begin{equation}\label{step3}
\begin{split}
&\frac{d}{dt}\left(H_{3}(t)+F_3(t)\right)+C\Big\{\|\chi_l(\gamma \tilde\rho^{\gamma-2}q-\phi)_{rr}\|^2
+\left\|\chi_l\left(\frac{dq}{dt}\right)_{rr}\right\|^2\Big\}\\[2mm]
&\leq\epsilon\|D^2q\|^2+C\Big\{\left\|\left(\frac{dq}{dt}\right)_{r}\right\|^2+\|q\|_{1}^2+\|q_t\|_1^2
+\|Du_t\|^2+\|Du\|_{1}^2
+\|\chi_{l}\partial_{r}\bar{\partial}Du\|^2\Big\}+C\delta\mathcal{D}(t)\\[2mm]
&\leq\epsilon\|D^2q\|^2+C\Big\{\left\|\left(\frac{dq}{dt}\right)_{r}\right\|^2+\|q\|_{1}^2+\|q_t\|_1^2
+\|Du_t\|^2+\|Du\|_{1}^2\Big\}+C\left\|\chi_l\bar{\partial}\frac{dq}{dt}\right\|^2_{1}+
C\delta\mathcal{D}(t)\\[2mm]
\end{split}
\end{equation}
where we have used Lemma \ref{lem10} with $k=1$, $m=0$ in the last inequality.
Then (\ref{step3}) and (\ref{step2-1}) imply that, for $\eta_3$ small, it holds
\begin{equation*}
\begin{split}
\frac{d}{dt}&\Big\{H_{1}(t)+F_1(t)+\eta_2H_2(t)+\eta_2F_{2}(t)+\eta_3H_3(t)+\eta_3F_3(t)\Big\}\\[2mm]
&+C\Big\{\|q\|_{1}^2+\|q_t\|_{1}^2+\|\nabla\phi\|^2+
\|D u\|_{1}^2+\|u_t\|_{1}^2+\left\|\frac{dq}{dt}\right\|^2_{1}\\[2mm]
&+\left\|\left(\frac{dq}{dt}\right)_t\right\|^2
+\|\chi_0D^{2}(\gamma\tilde{\rho}^{\gamma-2}q-\phi)\|^2+\|\chi_0D^3u\|^2
+\left\|\chi_0 D^2\frac{dq}{dt}\right\|^2\\[2mm]
&+\|\chi_l \bar{\partial}^2 D u\|^2+\left\|\chi_l\bar{\partial}^2 \frac{dq}{dt}\right\|^2
+\|\chi_l\bar{\partial}(\gamma \tilde\rho^{\gamma-2}q-\phi)_{r}\|^2
+\left\|\chi_l\bar{\partial}\left(\frac{dq}{dt}\right)_{r}\right\|^2\\[2mm]
&+\|\chi_l(\gamma \tilde\rho^{\gamma-2}q-\phi)_{rr}\|^2
+\left\|\chi_l\left(\frac{dq}{dt}\right)_{rr}\right\|^2\Big\}\\[2mm]
\leq &\epsilon \|D^2 q\|^2+C\delta\mathcal{D}(t),
\end{split}
\end{equation*}
that is
\begin{equation}\label{step3-1}
\begin{split}
\frac{d}{dt}&\Big\{H_{1}(t)+F_1(t)+\eta_2H_2(t)+\eta_2F_{2}(t)+\eta_3H_3(t)+\eta_3F_3(t)\Big\}\\[2mm]
&+C\Big\{\|q\|_{1}^2+\|q_t\|_{1}^2+\|\nabla\phi\|^2+
\|D u\|_{1}^2+\|u_t\|_{1}^2+\left\|\frac{dq}{dt}\right\|^2_{2}
+\left\|\left(\frac{dq}{dt}\right)_t\right\|^2
\Big\}\\[2mm]
\leq &\epsilon \|D^2 q\|^2+C\delta\mathcal{D}(t).
\end{split}
\end{equation}

On the other hand, using Lemma \ref{lemmstoke} with $k=1$, one has
\begin{equation*}
\|D^{3}u\|^2+\|D^{2}(\gamma\tilde\rho^{\gamma-2}q-\phi)\|^2\leq C\Big\{\left\|\frac{dq}{dt}\right\|_{2}^2+\|u\|_{2}^2+\|g^0\|_{2}^2+\|u_t\|_{1}^2+\|g\|_{1}^2\Big\},
\end{equation*}
which together with the elliptic estimate $\|D^2\phi\|^2\leq C\|q\|^2$
gives
\begin{equation*}
\begin{split}
\|D^2q\|^2&\leq C\left\{\|D^{2}(\gamma\tilde\rho^{\gamma-2}q)\|^2+\|q\|_1^2\right\}
\leq\left\{\|D^{2}(\gamma\tilde\rho^{\gamma-2}q-\phi)\|^2+\|q\|_1^2\right\}\\[2mm]
&\leq C\Big\{\left\|\frac{dq}{dt}\right\|_{2}^2+\|u\|_{2}^2+\|u_t\|_{1}^2+\|q\|_1^2\Big\}+C\delta\mathcal{D}(t),
\end{split}
\end{equation*}
Therefore, for $\epsilon$ small enough, (\ref{step3-1}) becomes
\begin{equation}\label{step3-2}
\begin{split}
\frac{d}{dt}&\Big\{H_{1}(t)+F_1(t)+\eta_2H_2(t)+\eta_2F_{2}(t)+\eta_3H_3(t)+\eta_3F_3(t)\Big\}\\[2mm]
&+C\Big\{\|q\|_{2}^2+\|q_t\|_{1}^2+\|\nabla\phi\|^2+
\|D u\|_{2}^2+\|u_t\|_{1}^2+\left\|\frac{dq}{dt}\right\|^2_{2}
+\left\|\left(\frac{dq}{dt}\right)_t\right\|^2
\Big\}\\[2mm]
\leq& C\delta\mathcal{D}(t).
\end{split}
\end{equation}

{\bf Step 4.} Adding the results on Lemma \ref {lemma6}, Lemma \ref{lem7} and  Lemma \ref{lem8}
with $k=2$, $m=0$, then for $\eta$ small, it holds that
\begin{equation}\label{step4}
\begin{split}
\frac{1}{2}&\frac{d}{dt}\Big\{
\int_{\Omega}\Big(\mu|D u_t|^2+(\mu+\lambda)|\mbox{div}u_t|^2\Big)dx
+\gamma\int_{\Omega}\chi_0^2\tilde{\rho}^{\gamma-4}|D^{3} q|^2dx\\[2mm]
&+\gamma\int_{\Omega}\chi_{l}^2\tilde\rho^{\gamma-2}|\bar{\partial}^3q|^2dx
+\int_{\Omega}\chi_{l}^2\tilde{\rho}|\bar{\partial}^3 u|^2dx +\int_{\Omega}\chi_{l}^2|\bar{\partial}^3D\phi|^2dx
+\eta\gamma\int_{\Omega}\chi_{l}^2\tilde\rho^{\gamma-4}|\bar{\partial}^{2}q_r|^2dx
\Big\}\\[2mm]
&+\frac{d}{dt}\Big\{-\int_{\Omega}\gamma\tilde{\rho}^{\gamma-1}q_t\mbox{div}u_tdx
-\int_{\Omega}\chi_0^2\tilde{\rho}^{-2} D_{ijk}  q \cdot D_{ijk} \phi dx\\
&+3\int_{\Omega}\chi_0^2\tilde{\rho}^{-2}D_{i}(\gamma\tilde{\rho}^{\gamma-2})\cdot D_{ijk}q D_{jk}q dx
+3\int_{\Omega}\chi_0^2\tilde{\rho}^{-2}D_{ij}(\gamma\tilde{\rho}^{\gamma-2})D_{ijk}q D_{k}q dx\\
&+\int_{\Omega}\chi_0^2\tilde{\rho}^{-2}D_{ijk}(\gamma\tilde{\rho}^{\gamma-2})D_{ijk}q q dx
+3\int_{\Omega}\chi_{l}^2\bar{\partial}(\gamma\tilde{\rho}^{\gamma-2}) \bar{\partial}^2q \bar{\partial}^3 q  dx\\
&+3\int_{\Omega}\chi_{l}^2\bar{\partial}^2(\gamma\tilde{\rho}^{\gamma-2})\bar{\partial} q \bar{\partial}^3 q  dx
+\int_{\Omega}\chi_{l}^2\bar{\partial}^3(\gamma\tilde{\rho}^{\gamma-2}) q \bar{\partial}^3 q  dx\\[2mm]
&-\int_{\Omega}\chi_{l}^2\tilde\rho^{-2}\bar{\partial}^{k}\partial_{r}^{m}q_{r}\bar{\partial}^{k}\partial_{r}^{m}\phi_{r}dx
-\int_{\Omega}\chi_{l}^2 \tilde\rho^{-2}\bar{\partial}^{2}q_{r}\Big(
(\gamma\tilde\rho^{\gamma-2})_r\bar{\partial}^{2}q
+2\bar{\partial}(\gamma\tilde{\rho}^{\gamma-2})\bar{\partial}q_{r}\\[2mm]
&+\bar{\partial}^{2}(\gamma\tilde\rho^{\gamma-2})q_{r}+\bar{\partial}(\gamma\rho^{\gamma-2})_{r}\bar{\partial}q
+\bar{\partial}^{2}(\gamma\tilde\rho^{\gamma-2})_rqdx\Big)dx
\Big\}\\[2mm]
&+C\Big\{ \|q_{tt}\|^2+\|u_{tt}\|^2+
\|\chi_0D^{3}(\gamma\tilde{\rho}^{\gamma-2}q-\phi)\|^2
+\|\chi_0D^4u\|^2+\left\|\chi_0D^{3}\frac{dq}{dt}\right\|^2\\[2mm]
&+\|\chi_l \bar{\partial}^3 D u\|^2+\|\chi_l\bar{\partial}^3 \frac{dq}{dt}\|^2
+\|\chi_l\bar{\partial}^{2} (\gamma \tilde\rho^{\gamma-2}q-\phi)_{r}\|^2
+\left\|\chi_l\bar{\partial}^{2}\left(\frac{dq}{dt}\right)_{r}\right\|^2
\Big\}\\[2mm]
\leq& \epsilon \|D^3 q\|^2+C\Big\{\|Du\|_{2}^2+\|u_t\|_{2}^2+\|q\|_{2}^2+\|q_t\|_2^2+\left\|\left(\frac{dq}{dt}\right)_{r}\right\|_1^2
\Big\}+C\delta\mathcal{D}(t).
\end{split}
\end{equation}
Denoting the $t-$ derivative  of (\ref{step4}) by  $\frac{d}{dt}H_{4}(t)+\frac{d}{dt}F_4(t)$, then putting (\ref{step4}) together with (\ref{step3-2}), it implies that, for $\eta_4$ small,
\begin{equation}\label{step4-1}
\begin{split}
\frac{d}{dt}&\Big\{H_{1}(t)+F_1(t)+\sum_{i=2}^{4}\left(\eta_{i}H_{i}(t)+\eta_{i}F_{i}(t)\right)\Big\}\\[2mm]
&+C\Big\{\|q\|_{2}^2+\|q_t\|_{1}^2+\|D\phi\|^2+
\|D u\|_{2}^2+\|u_t\|_{1}^2+\|q_{tt}\|^2+\|u_{tt}\|^2+\left\|\frac{dq}{dt}\right\|^2_{2}
+\left\|\left(\frac{dq}{dt}\right)_t\right\|^2
\Big\}\\[2mm]
&+\|\chi_0D^{3}(\gamma\tilde{\rho}^{\gamma-2}q-\phi)\|^2
+\|\chi_0D^4u\|^2+\left\|\chi_0D^{3}\frac{dq}{dt}\right\|^2
+\|\chi_l \bar{\partial}^3 D u\|^2+\|\chi_l\bar{\partial}^3 \frac{dq}{dt}\|^2\\[2mm]
&+\|\chi_l\bar{\partial}^{2} (\gamma \tilde\rho^{\gamma-2}q-\phi)_{r}\|^2
+\left\|\chi_l\bar{\partial}^{2}\left(\frac{dq}{dt}\right)_{r}\right\|^2
\Big\}\\[2mm]
\leq& \epsilon \|D^3 q\|^2+C\delta\mathcal{D}(t).
\end{split}
\end{equation}

{\bf Step 5.} Lemma \ref{lem8} with $k=1$,  $m=1$ implies that
\begin{equation}\label{step5}
\begin{split}
&\frac{d}{dt}\left(H_5(t)+F_5(t)\right)
+C\Big\{\|\chi_l\bar{\partial}\partial_{r} (\gamma \tilde\rho^{\gamma-2}q-\phi)_{r}\|^2
+\left\|\chi_l\bar{\partial}\partial_{r}\left(\frac{dq}{dt}\right)_{r}\right\|^2\Big\}\\[2mm]
&\leq\epsilon\|D^3q\|^2+C\left(\left\|\left(\frac{dq}{dt}\right)_{r}\right\|_1^2+\|q\|_2^2+\|q_t\|_2^2+\|D^2u_t\|^2+\|Du\|_{2}^2
+\|\chi_{l}\bar{\partial}^2\partial_{r}Du\|^2\right)+C\delta\mathcal{D}(t)\\[2mm]
&\leq \epsilon\|D^3q\|^2+C\Big\{\left\|\left(\frac{dq}{dt}\right)_{r}\right\|_1^2+\|q\|_2^2+\|q_t\|_2^2+\|u_t\|_2^2+\|Du\|_{2}^2
+\left\|\chi_l\bar{\partial}^{2}\frac{dq}{dt}\right\|^2_{1}\Big\}+C\delta\mathcal{D}(t),
\end{split}
\end{equation}
where we have used Lemma \ref{lem10} with $k=2$, $m=0$.
Adding (\ref{step5}) and (\ref{step4-1}) implies that, for $\eta_5$ small, it holds
\begin{equation}\label{step5-1}
\begin{split}
\frac{d}{dt}&\Big\{H_{1}(t)+F_1(t)+\sum_{i=2}^{5}\left(\eta_iH_i(t)+\eta_iF_{i}(t)\right)\Big\}\\[2mm]
&+C\Big\{\|q\|_{2}^2+\|q_t\|_{1}^2+\|D\phi\|^2+
\|D u\|_{2}^2+\|u_t\|_{1}^2+\|q_{tt}\|^2+\|u_{tt}\|^2+\left\|\frac{dq}{dt}\right\|^2_{2}
+\left\|\left(\frac{dq}{dt}\right)_t\right\|^2\\[2mm]
&+\|\chi_0D^{3}(\gamma\tilde{\rho}^{\gamma-2}q-\phi)\|^2
+\|\chi_0D^4u\|^2+\left\|\chi_0D^{3}\frac{dq}{dt}\right\|^2
+\|\chi_l \bar{\partial}^3 D u\|^2+\left\|\chi_l\bar{\partial}^3 \frac{dq}{dt}\right\|^2\\[2mm]
&+\|\chi_l\bar{\partial}^{2} (\gamma \tilde\rho^{\gamma-2}q-\phi)_{r}\|^2
+\left\|\chi_l\bar{\partial}^{2}\left(\frac{dq}{dt}\right)_{r}\right\|^2
+\|\chi_l\bar{\partial} (\gamma \tilde\rho^{\gamma-2}q-\phi)_{rr}\|^2
+\left\|\chi_l\bar{\partial}\left(\frac{dq}{dt}\right)_{rr}\right\|^2
\Big\}\\[2mm]
\leq &\epsilon \|D^3 q\|^2+C\delta\mathcal{D}(t).
\end{split}
\end{equation}

{\bf Step 6.} It is obvious to see that
\begin{equation}\label{step6}
\begin{split}
&\frac{d}{dt}\left(H_{6}(t)+F_6(t)\right)
+C\Big\{\|\chi_l\partial_{r}^{3} (\gamma \tilde\rho^{\gamma-2}q-\phi)\|^2
+\left\|\chi_l\partial_{r}^{3}\left(\frac{dq}{dt}\right)\right\|^2\Big\}\\[2mm]
&\leq\epsilon\|D^3q\|^2+C\left(\left\|\left(\frac{dq}{dt}\right)_{r}\right\|_1^2+\|q\|_2^2+\|q_t\|_2^2+\|D^2u_t\|^2+\|Du\|_{2}^2
+\|\chi_{l}\partial_{r}^{2}\bar{\partial}Du\|^2\right)+C\delta\mathcal{D}(t)\\[2mm]
&\leq \epsilon\|D^3q\|^2+C\Big\{\left\|\left(\frac{dq}{dt}\right)_{r}\right\|_1^2+\|q\|_2^2+\|q_t\|_2^2+\|u_t\|_{2}^2+\|Du\|_{2}^2
+\left\|\chi_l\bar{\partial}^{2}\frac{dq}{dt}\right\|^2_{2}\Big\}+C\delta\mathcal{D}(t),
\end{split}
\end{equation}
from Lemma \ref{lem8} with $k=0$,  $m=2$, where we have used Lemma \ref{lem10} with $k=1$, $m=1$. Therefore, (\ref{step6}) and (\ref{step5-1}) imply that, for $\eta_6$ small, it holds
\begin{equation*}
\begin{split}
\frac{d}{dt}&\Big\{H_{1}(t)+F_1(t)+\sum_{i=2}^{6}(\eta_iH_i(t)+\eta_iF_{i}(t))\Big\}\\[2mm]
&+C\Big\{\|q\|_{2}^2+\|q_t\|_{1}^2+\|\nabla\phi\|^2+
\|D u\|_{2}^2+\|u_t\|_{1}^2+\|q_{tt}\|^2+\|u_{tt}\|^2+\left\|\frac{dq}{dt}\right\|^2_{2}
+\left\|\left(\frac{dq}{dt}\right)_t\right\|^2\\[2mm]
&+\|\chi_0D^{3}(\gamma\tilde{\rho}^{\gamma-2}q-\phi)\|^2
+\|\chi_0D^4u\|^2+\left\|\chi_0D^{3}\frac{dq}{dt}\right\|^2
+\|\chi_l \bar{\partial}^3 \nabla u\|^2+\left\|\chi_l\bar{\partial}^3 \frac{dq}{dt}\right\|^2\\[2mm]
&+\|\chi_l\bar{\partial}^{2} (\gamma \tilde\rho^{\gamma-2}q-\phi)_{r}\|^2
+\left\|\chi_l\bar{\partial}^{2}\left(\frac{dq}{dt}\right)_{r}\right\|^2
+\|\chi_l\bar{\partial} (\gamma \tilde\rho^{\gamma-2}q-\phi)_{rr}\|^2
+\left\|\chi_l\bar{\partial}\left(\frac{dq}{dt}\right)_{rr}\right\|^2\\[2mm]
&+\|\chi_l\partial_{r}^{3} (\gamma \tilde\rho^{\gamma-2}q-\phi)\|^2
+\left\|\chi_l\partial_{r}^{3}\left(\frac{dq}{dt}\right)\right\|^2
\Big\}\\[2mm]
\leq& \epsilon \|D^3 q\|^2+C\delta\mathcal{D}(t),
\end{split}
\end{equation*}
which yields that
\begin{equation}\label{1step6-1}
\begin{split}
\frac{d}{dt}&\Big\{H_{1}(t)+F_1(t)+\sum_{i=2}^{6}(\eta_iH_i(t)+\eta_iF_{i}(t))\Big\}\\[2mm]
&+C\Big\{\|q\|_{2}^2+\|q_t\|_{1}^2+\|\nabla\phi\|^2+
\|D u\|_{2}^2+\|u_t\|_{1}^2+\|q_{tt}\|^2+\|u_{tt}\|^2+\left\|\frac{dq}{dt}\right\|^2_{3}
+\left\|\left(\frac{dq}{dt}\right)_t\right\|^2\\[2mm]
&+\|\chi_0D^{3}(\gamma\tilde{\rho}^{\gamma-2}q-\phi)\|^2
+\|\chi_0D^4u\|^2
+\|\chi_l \bar{\partial}^3 \nabla u\|^2\\[2mm]
&+\|\chi_l\bar{\partial}^{2} (\gamma \tilde\rho^{\gamma-2}q-\phi)_{r}\|^2
+\|\chi_l\bar{\partial} (\gamma \tilde\rho^{\gamma-2}q-\phi)_{rr}\|^2
+\|\chi_l\partial_{r}^{3} (\gamma \tilde\rho^{\gamma-2}q-\phi)\|^2
\Big\}\\[2mm]
\leq& \epsilon \|D^3 q\|^2+C\delta\mathcal{D}(t).
\end{split}
\end{equation}
On the other hand, using Lemma \ref{lemmstoke} with $k=2$, one has
\begin{equation*}
\|D^{4}u\|^2+\|D^{3}(\gamma\tilde\rho^{\gamma-2}q-\phi)\|^2\leq C\Big\{\left\|\frac{dq}{dt}\right\|_{3}+\|u\|_{3}+\|g^0\|_{3}+\|u_t\|_{2}+\|g\|_{2}\Big\},
\end{equation*}
which together with the elliptic estimate $\|D^3\phi\|^2\leq C\|q\|_{1}^2$
gives
\begin{equation*}
\begin{split}
\|D^3q\|^2&\leq C\left\{\|D^{3}(\gamma\tilde\rho^{\gamma-2}q)\|^2+\|q\|_2^2\right\}
\leq\left\{\|D^{3}(\gamma\tilde\rho^{\gamma-2}q-\phi)\|^2+\|q\|_2^2\right\}\\[1mm]
&\leq C\Big\{\left\|\frac{dq}{dt}\right\|_{3}^2+\|u\|_{3}^2+\|u_t\|_{3}^2+\|q\|_{2}^2\Big\}+C\delta\mathcal{D}(t).
\end{split}
\end{equation*}
Therefore, for $\epsilon$ small enough, (\ref{1step6-1}) is controlled as
\begin{equation}\label{step6-1}
\begin{split}
\frac{d}{dt}&\Big\{H_{1}(t)+F_1(t)+\sum_{i=2}^{6}(\eta_iH_i(t)+\eta_iF_{i}(t))\Big\}\\[2mm]
&+C\Big\{\|q\|_{3}^2+\|q_t\|_{1}^2+\|\nabla\phi\|^2+
\|D u\|_{3}^2+\|u_t\|_{1}^2+\|q_{tt}\|^2+\|u_{tt}\|^2+\left\|\frac{dq}{dt}\right\|^2_{3}
\Big\}\leq C\delta\mathcal{D}(t).
\end{split}
\end{equation}

{\bf Step 7.} Let
\begin{equation*}
\widetilde{\mathcal{E}}(t)\equiv H_{1}(t)+F_1(t)+\sum_{i=2}^{6}\left(\eta_iH_i(t)+\eta_iF_{i}(t)\right),
\end{equation*}
and
\begin{equation*}
\widetilde{\mathcal{D}}(t)\equiv\|q\|_{3}^2+\|q_t\|_{1}^2+\|\nabla\phi\|^2+
\|D u\|_{3}^2+\|u_t\|_{1}^2+\|q_{tt}\|^2+\|u_{tt}\|^2+\left\|\frac{dq}{dt}\right\|^2_{3},
\end{equation*}
then we obtian
\begin{equation}\label{last3}
\frac{d}{dt}\widetilde{\mathcal{E}}(t)+C\widetilde{\mathcal{D}}(t)\leq  C\delta\mathcal{D}(t).
\end{equation}
Recalling the definitions of $\mathcal{E}(t)$, $\mathcal{D}(t)$, it is directly to see that
\begin{equation}\label{last4}
\mathcal{E}(t)\leq\mathcal{D}(t).
\end{equation}
Now we claim  that
\begin{equation}\label{last1}
\mathcal{D}(t)\leq C\widetilde{\mathcal{D}}(t)+C\delta\mathcal{D}(t),
\end{equation}
which implies for $\delta$ small, it holds
\begin{equation}\label{last2}
\mathcal{D}(t)\leq C\widetilde{\mathcal{D}}(t).
\end{equation}
Indeed, we can show that
\begin{equation*}
\|D^2q_{t}\|^2\leq C\Big\{\left\|D^{2}\frac{dq}{dt}\right\|^2+\|D^{2}(u\cdot\nabla q)\|^2\Big\}\leq C\widetilde{\mathcal{D}}(t)+C\delta\mathcal{D}(t).
\end{equation*}
and
\begin{equation*}
\begin{split}
\|D^{2}u_t\|^2&\leq C\Big\{\|u_{tt}\|^2+\|\nabla q_t\|^2+\|\nabla\phi_t\|^2+\| u_t\|^2+\|f_t\|^2+\|\nabla u_t\|^2\Big\}\\[2mm]
&\leq C\Big\{\|u_{tt}\|^2+\|\nabla q_t\|^2+\|u\|^2+\| u_t\|^2+\|f_t\|^2+\|\nabla u_t\|^2\Big\}\\[2mm]
&\leq C\widetilde{\mathcal{D}}(t)+C\delta\mathcal{D}(t).
\end{split}
\end{equation*}
Therefore, the claim (\ref{last1}) is proved.

On the other hand, H$\ddot{o}$lder's inequality implies that
\begin{equation}\label{last5}
C^{-1}\mathcal{E}_1(t)\leq \widetilde{\mathcal{E}}(t)\leq C\mathcal{E}_1(t).
\end{equation}
where
\begin{equation*}
\mathcal{E}_1(t)=\|q\|_{3}^2+\|q_t\|^2+\|u\|_1^2+\|u_t\|_1^2+\|D\phi\|^2.
\end{equation*}
Notice that Lemma \ref{lem-neu} yields that
\begin{equation*}
\|\nabla\phi\|_{4}^2\leq C\|q\|_{3}^2.
\end{equation*}
By using Lemma \ref{lemell}, one has
\begin{equation*}
\|D^2u\|^2+\|D^3u\|^2\leq C\mathcal{E}_1(t)+C\delta\mathcal{D}(t).
\end{equation*}
Then, in view of the  equation $(\ref{per1})_1$, we can obtain the estimate of $\|\nabla q_t\|^2$ and $\|D^{2}q_t\|^2$. Therefore,
\begin{equation*}
C^{-1}\mathcal{E}(t)\leq \mathcal{E}_1(t)\leq \mathcal{E}(t).
\end{equation*}
Then, putting (\ref{last4}), (\ref{last2})  and  (\ref{last5}) into (\ref{last3}), we obtain
\begin{equation*}
\frac{d}{dt}\widetilde{\mathcal{E}}(t)+\sigma\widetilde{\mathcal{E}}(t)\leq 0,
\end{equation*}
which gives
\begin{equation*}
\widetilde{\mathcal{E}}(t)\leq e^{-\sigma t}\widetilde{\mathcal{E}}(0).
\end{equation*}
Then
\begin{equation*}
\mathcal{E}(t)\leq C e^{-\sigma t}\mathcal{E}(0).
\end{equation*}
The proof is completed.
\end{proof}

\centerline{\bf Acknowledgements}
Liu's research is supported by National Natural Science
Foundation of China (No.11926418).
The authors are grateful to Professor Tao Luo for helpful suggestions and discussions.

\end{document}